\documentclass[12pt]{article}
\usepackage[utf8]{inputenc}
\usepackage{amsmath}
\usepackage{amssymb}
\usepackage{mathrsfs}
\usepackage{amsthm}
\usepackage{cite}
\usepackage{hyperref}
\usepackage{enumerate}
\usepackage{combelow}
\usepackage{babel}
\usepackage[left=0.7in,right=0.7in,top=0.7in,bottom=0.7in]{geometry}
\usepackage{titlesec}
\titleformat*{\section}{\large\bfseries}
\newtheorem{theorem}{Theorem}[section]
\newtheorem{lemma}[theorem]{Lemma}
\newtheorem{corollary}[theorem]{Corollary}
\newtheorem{definition}[theorem]{Definition}

\numberwithin{equation}{section}
\title{Optimum solution to a boundary-value problem for an infinite system of differential equations in $c_0$ and $\ell_1$ via measure of noncompactness}
\author{\large  Parveen Kumar$^1$ and Ankit Kumar$^2$\footnote{Corresponding author}\\ 
{\small $^1,^2$Department of Mathematics, Punjab Engineering College (Deemed to be University)}\\
{\small Chandigarh-160012, India}\\ 
{\small E-mail: $^1$parvveenbeniwal11@gmail.com, $^2$ankitkumar@pec.edu.in}}
\date{}
\begin{document}
\maketitle
\begin{abstract}
We propose a new class of cyclic (noncyclic) operator type  condensing mappings and establish conditions that guarantee the existence of best proximity points (pairs) using measure of noncom-pactness. The  results obtained are further applied to investigate the optimum solution of an infinite system of differential equations in $c_0$ and $\ell_1$.
\end{abstract}
\section{Introduction}
Consider two nonempty subsets $\Delta$ and $\Omega$ of a Banach space $\mathfrak{X}$. Define 
\begin{align*}
dist(\Delta,\Omega)&=\inf\{\Vert \varrho-\varsigma \Vert :\varrho \in \Delta \mbox{ and } \varsigma \in \Omega\},\\
\Delta_0&=\{\varrho\in \Delta:\Vert \varrho-\varsigma\Vert=dist(\Delta,\Omega) \mbox{ for some } \varsigma \in \Omega\},\\
\Omega_0&=\{\varsigma\in \Omega:\Vert \varrho-\varsigma\Vert =dist(\Delta,\Omega) \mbox{ for some } \varrho \in \Delta\}.
\end{align*}
A pair $(\Delta,\Omega)$ possesses a property whenever both $\Delta$ and $\Omega$ possess that property. Denote by $\overline{\Delta}$, the closure of the set $\Delta$ and $\overline{\operatorname{con}}(\Delta)$, the closed convex hull of the set $\Delta$. Whenever $(\Delta,\Omega)$ is a  convex, nonempty and weakly compact pair in a Banach space $\mathfrak{X}$, then the pair $(\Delta_0,\Omega_0)$ also share these properties.
\begin{definition}
\emph{A nonempty pair $(\Delta,\Omega)$ in a normed linear space $\mathfrak{X}$ is termed proximinal whenever $\Delta_0=\Delta$ and $\Omega_0=\Omega$. }
\end{definition}
\begin{lemma}
\cite{1} Let $\Delta$ be a compact and nonempty subset of a Banach space $\mathfrak{X}$. Then $\overline{\operatorname{con}}(\Delta)$ is compact. 
\end{lemma}
The following classical result was established by Schauder in 1930:
\begin{theorem}
\cite{2} Let $\mathfrak{X}$ be a Banach space containing  a nonempty, convex and compact subset $\Delta$. Suppose that  $\daleth:\Delta \rightarrow \Delta$ is a continuous mapping. Then $\daleth$ admits a fixed point.
\end{theorem}
\begin{definition}
\emph{Let $\Omega$ and $\Delta$ be normed linear spaces with $\Gamma \subseteq \Delta$. An operator $\daleth:\Gamma \to \Omega$ is termed a compact operator if $\daleth$ carries bounded sets into relatively compact sets and $\daleth$ is continuous.}
\end{definition}
We now recall a classical generalization of Schauder’s fixed point theorem.
\begin{theorem}
\cite{2} Let $\mathfrak{X}$ be a Banach space containing a  convex, nonempty, bounded and closed subset $\Delta$. Suppose that  $\daleth:\Delta \rightarrow \Delta$ is a compact operator. Then $\daleth$ admits a fixed point.
\end{theorem}
Kuratowski \cite{3} made the initial contribution to the development of the concept of measure of noncompactness. Let $\Pi$ be the collection of all  subsets of a metric space $(\mathfrak{X},d)$ which are bounded and nonempty. 
\begin{definition}
\emph{\cite{3} Define the function $\pi: \Pi \rightarrow [0,\infty)$ defined by $$\pi(\Delta) = \inf\{\delta > 0 : \Delta \text{ admits a finite cover by sets of diameter} \leq \delta\} \text{ for all } \Delta \in \Pi.$$  This function $\pi$ is termed the Kuratowski measure of noncompactness.}
\end{definition}
Gohberg et al. \cite{4} extended the notion of Kuratowski measure of noncompactness.
\begin{definition}
\emph{Define the function $\Upsilon : \Pi \to [0,\infty)$ defined by $$\Upsilon(\Delta) = \inf\{\delta > 0 : \Delta \text{ admits a finite cover by finitely many balls with radii } \leq \delta\} \text{ for all }\Delta \in \Pi.$$ This function $\Upsilon$ is termed the Hausdorff measure of noncompactness.}
\end{definition}
\begin{definition}
\emph{A function $\Xi:\Theta \rightarrow [0,\infty)$ is called a measure of noncompactness (MNC) if}

\emph{(i) $\Xi(\Delta) = \Xi(\overline{\Delta})$ for all $\Delta \in \Theta$,}

\emph{(ii)  $\Delta$ is relatively compact if and only if $\Xi(\Delta) = 0$,}

\emph{(iii) $\Xi(\Delta \cup \Omega) = \max\{\Xi(\Delta), \Xi(\Omega)\}$ for all $\Delta,\Omega \in \Theta$.}
\end{definition}
Also, we have:

(i) If $\Delta \subseteq \Omega$, then $\Xi(\Delta) \leq \Xi(\Omega)$,

(ii) If $\Delta$ is a finite set, then $\Xi(\Delta) = 0$,

(iii) $\Xi(\Delta \cap \Omega) \le \min\{\Xi(\Delta), \Xi(\Omega)\}$ for all $\Delta, \Omega \in \Theta$,

(iv) Whenever $\{\Delta_n\}$ is a decreasing sequence of  bounded, nonempty and closed subsets of $\mathfrak{X}$ for which $\lim\limits_{n \to \infty} \Xi(\Delta_n) = 0$, then $\Delta_{\infty}= \bigcap\limits_{n \geq 1} \Delta_n$ is compact and nonempty.

Furthermore, assuming that $\mathfrak{X}$ is a Banach space, then

(i) $\Xi(\Delta + \Omega) \leq \Xi(\Delta) + \Xi(\Omega)$ for all $\Delta,\Omega \in \Theta$,

(ii) For any number $\alpha$ and any $\Delta \in \Theta$, $\Xi(\alpha \Delta) = \vert \alpha \vert \Xi(\Delta)$,

(iii) $\Xi(\operatorname{\overline{con}}(\Delta)) = \Xi(\Delta)$ for all $\Delta \in \Theta$.\\
A thorough exposition of measure of noncompactness is presented in \cite{5}. Building on this framework, Darbo \cite{6} extended Schauder’s fixed point theorem to the broader class of condensing operators.
\begin{theorem}
\cite{6} Let $\mathfrak{X}$ be a Banach space equipped with a measure of noncompactness $\Upsilon$ containing a  convex, nonempty, bounded and closed subset $\Delta$. Suppose that $\daleth : \Delta \rightarrow \Delta$ is a continuous mapping satisfying $\Xi(\daleth(\Gamma)) \leq \eta \Xi(\Gamma)$ for some $\eta \in [0,1)$ and for all  bounded and nonempty subsets $\Gamma$ of $\Delta$. Then $\daleth$ admits a fixed point.
\end{theorem}
A Banach space $\mathfrak{X}$ is called strictly convex if for all $\varrho, \varsigma, \zeta \in \mathfrak{X}$, $\varrho \neq \varsigma$ and $\beta>0$, we have $\Vert \varrho-\zeta \Vert \leq \beta$, $\Vert \varsigma-\zeta\Vert \leq \beta$  implies that $\left\Vert \frac{\varrho+\varsigma}{2}-\zeta \right\Vert <\beta$. 

A mapping $\daleth:\Delta \cup \Omega \rightarrow \Delta \cup \Omega$ is termed cyclic relatively nonexpansive mapping if $\daleth(\Delta)\subseteq \Omega$, $\daleth(\Omega)\subseteq \Delta$ and $\Vert \daleth\varrho-\daleth\varsigma\Vert \leq \Vert \varrho-\varsigma\Vert$ for all $(\varrho,\varsigma)\in \Delta\times \Omega$. We seek to identify sufficient conditions on the mapping $\daleth$ and $(\Delta,\Omega)$ guaranteeing the existence of a best proximity point of $\daleth$. A point $\xi \in \Delta \cup \Omega$ is called a best proximity point of $\daleth$ if $\Vert \xi-\daleth \xi\Vert=dist(\Delta,\Omega)$. A mapping $\daleth:\Delta \cup \Omega \rightarrow \Delta \cup \Omega$ is termed  noncyclic relatively nonexpansive mapping if $\daleth(\Delta)\subseteq \Delta$, $\daleth(\Omega)\subseteq \Omega$ and $\Vert \daleth \varrho-\daleth \varsigma\Vert \leq \Vert \varrho-\varsigma\Vert$ for all $(\varrho,\varsigma)\in \Delta\times \Omega$. A point $(\xi,\zeta) \in \Delta \cup \Omega$ is called a best proximity pair of $\daleth$ if $\daleth \xi=\xi$, $\daleth \zeta=\zeta$ and $\Vert \xi-\zeta \Vert=dist(\Delta,\Omega)$. 

The following results, established by Eldred et al. \cite{11}, will be  instrumental in deriving the results of this paper:
\begin{theorem}\label{theorem1}
\cite{11} Let $\mathfrak{X}$ be a Banach space containing  a  compact, convex and nonempty pair $(\Delta,\Omega)$. Suppose that $ \daleth:\Delta \cup \Omega \rightarrow \Delta \cup \Omega$ is a cyclic relatively nonexpansive mapping. Then $\daleth$ possesses a best proximity point.
\end{theorem}
\begin{theorem}\label{theorem2}
\cite{11} Let $\mathfrak{X}$ be a strictly convex Banach space containing a compact, convex and nonempty pair $(\Delta,\Omega)$. Suppose that $\daleth:\Delta \cup \Omega \rightarrow \Delta \cup \Omega$ is a noncyclic relatively nonexpansive mapping. Then $\daleth$ possesses a best proximity pair.
\end{theorem}
In recent years, several authors have studied best proximity point results and explored their applications in diverse areas. Gabeleh and Markin \cite{7} developed the notion of cyclic (noncyclic) condensing operators and demonstrated the existence of best proximity points (best proximity pairs) by employing measure of noncompactness. 
\begin{theorem} 
\cite[Theorem 3.2]{7} Let $\mathfrak{X}$ be a Banach space containing a  closed, nonempty, convex and bounded pair $(\Delta,\Omega)$. Suppose that $\daleth:\Delta \cup \Omega \rightarrow \Delta \cup \Omega$ is a cyclic relatively nonexpansive mapping. Assume that $\Delta_0$ is nonempty and $\daleth$ is compact. Then $\daleth$ has a best proximity point.
\end{theorem}
\begin{definition}\label{definition 1.1}
\emph{\cite[Definition 3.3]{7}  Let $\mathfrak{X}$ be a Banach space equipped with a measure of noncompact-ness $\Xi$ containing a convex and nonempty pair $(\Delta,\Omega)$. A cyclic mapping $\daleth:\Delta \cup \Omega \rightarrow \Delta \cup \Omega$ is called a condensing operator if for any  proximinal, nonempty, $\daleth$-invariant, convex, closed and bounded pair $(\omega,\Lambda)\subseteq (\Delta,\Omega)$ such that $dist(\omega,\Lambda)=dist(\Delta,\Omega)$, there exists $\beta \in (0,1)$ such that $\Xi(\daleth(\omega) \cup \daleth(\Lambda)) \leq \beta \Xi(\omega \cup \Lambda)$. } 
\end{definition}
\begin{theorem}
\cite[Theorem 3.4]{7} Let $\mathfrak{X}$ be a Banach space equipped with a measure of noncompactness $\Xi$ containing  a  closed, nonempty, convex and bounded pair $(\Delta,\Omega)$. Suppose that $\daleth:\Delta \cup \Omega \rightarrow \Delta \cup \Omega$ is a cyclic relatively nonexpansive mapping which fulfills the condensing property according to Definition \ref{definition 1.1}. Assume that $\Delta_0$ is nonempty. Then $\daleth$ has a best proximity point.  
\end{theorem}
\begin{theorem}
\cite[Theorem 4.1]{7} Let $\mathfrak{X}$ be a strictly convex Banach space containing a  convex, nonempty, closed and bounded pair $(\Delta,\Omega)$. Suppose  that $\daleth:\Delta \cup \Omega \rightarrow \Delta \cup \Omega$ is a noncyclic relatively nonexpansive mapping. Assume that $\Delta_0$ is nonempty, $\daleth \vert_{\Delta}$ and $\daleth \vert_{\Omega}$ are compact. Then $\daleth$ has a best proximity pair.
\end{theorem}
\begin{definition}\label{definition 1.2}
\emph{\cite[Definition 4.2]{7} Let $\mathfrak{X}$ be a Banach space equipped with a measure of noncompact-ness $\Xi$ containing a convex and nonempty pair $(\Delta,\Omega)$. A noncyclic mapping $\daleth:\Delta \cup \Omega \rightarrow \Delta \cup \Omega$ is called a condensing operator if for any  proximinal, nonempty, $\daleth$-invariant, convex, closed and bounded pair $(\omega,\Lambda)\subseteq (\Delta,\Omega)$ such that $dist(\omega,\Lambda)=dist(\Delta,\Omega)$, there exists $\eta \in (0,1)$ such that $\Xi(\daleth(\omega) \cup \daleth(\Lambda)) \leq \eta \Xi(\omega \cup \Lambda)$. } 
\end{definition}
\begin{theorem}
\cite[Theorem 4.3]{7} Let $\mathfrak{X}$ be a strictly convex Banach space equipped with a measure of noncompactness $\Xi$ containing a closed, nonempty, bounded and convex  pair  $(\Delta,\Omega)$. Suppose that $\daleth:\Delta \cup \Omega \rightarrow \Delta \cup \Omega$ is a noncyclic relatively nonexpansive mapping that fulfills the condensing property according to Definition  \ref{definition 1.2}. Assume that $\Delta_0$ is nonempty. Then $\daleth$  admits a best proximity pair.  
\end{theorem}
Several authors \cite{12,13,14,16,18, n13} have recently proposed various condensing operators and derived corresponding best proximity point and best proximity pair results by means of measure of noncompa-ctness. The results served as a basis for establishing the existence of solutions to systems of differential and integro-differential equations. Moreover, various researchers \cite{n6,n7,n8,n9, n10, n11, n12}  have explored the application of Darbo's fixed point theorem and its generalizations for solving infinite system of differential equations in a variety of sequence spaces. 
 
Drawing inspiration from these developments,  we formulate a new class of  cyclic (noncyclic) operator type condensing mappings and derive  conditions ensuring the existence of best proximity point (pair) of these mappings by employing the framework of measure of noncompactness. Furthermore, we explore a novel application of our results to establish the optimal solution of an infinite system of second order differential equations, subject to boundary conditions in $c_0$ and $\ell_1$ space.
\section{Main Results}
The discussion in this section is dedicated to establish best proximity points (pairs) theorems for cyclic (noncyclic) operator type condensing mappings by using the notion of measure of noncompactness. We commence with the following definitions:
\begin{definition}\label{definition 2.1}
\emph{\cite{n3} Let $\theta:[0,\infty) \rightarrow [1,\infty)$ be a mapping satisfying}

\emph{(i) $\theta$ is strictly increasing and continuous function,}

\emph{(ii) corresponding to each sequence $\{\varrho_n\} \subseteq (0,\infty)$, we have}
$$\lim_{n\rightarrow \infty} \theta(\varrho_n)=1 \emph{ if and only if } \lim_{n \rightarrow \infty} \varrho_n=0.$$
\end{definition}
For example, $\theta(\varrho)=e^{\varrho}$. The collection of these mappings is denoted by $\Theta$.
\begin{definition}\label{definition 2.2}
\emph{Let $\psi:[1,\infty) \rightarrow [0,\infty)$ be a mapping satisfying}

\emph{(i) $\psi(1)=0$,}

\emph{(ii) $\psi$ is continuous,}

\emph{(iii) corresponding to each sequence $\{\varrho_n\} \subseteq (1,\infty)$, we have}
$$\lim_{n\rightarrow \infty} \psi(\varrho_n)=0 \emph{ if and only if } \lim_{n \rightarrow \infty} \varrho_n=1.$$
\end{definition}
The collection of these mappings is denoted by $\Psi$. Some examples are:

(i) $\psi(\varrho)=\varrho-\varrho^{\alpha}$, where $\alpha \in [0,1)$,

(ii) $\psi(\varrho)=\ln \varrho$,

(iii) $\psi(\varrho)=e^{\varrho-1}-1$.
\begin{definition}\label{definition 2.3}
\emph{\cite{n4} Let $\phi:[0,\infty) \rightarrow [0,\infty)$ be a mapping satisfying}

\emph{(i) $\phi$ is nondecreasing and continuous,}

\emph{(ii) corresponding to each sequence $\{\varrho_n\} \subseteq [0,\infty)$, we have}
$$\lim_{n\rightarrow \infty} \phi(\varrho_n)=0 \emph{ if and only if } \lim_{n \rightarrow \infty} \varrho_n=0.$$
\end{definition}
The collection of these mappings is denoted by $\Phi$.
\begin{definition}\label{definition 2.4}
\emph{\cite{n2} Let $\varkappa:\mathbb{R}^+ \times \mathbb{R}^+ \rightarrow \mathbb{R}^+$ be a mapping satisfying }

\emph{(i) $\varkappa$ is nondecreasing and continuous,}

\emph{(ii) $\max\{\varrho,\varsigma\} \leq \varkappa(\varrho,\varsigma)$ for $\varrho,\varsigma \geq 0$,}

\emph{(iii) $\varkappa(\varrho+\varsigma,\zeta+\varpi) \leq \varkappa(\varrho,\zeta)+\varkappa(\varsigma,\varpi)$.}
\end{definition}
For example, $\varkappa(\varrho,\varsigma)=\varrho+\varsigma$. The collection of these mappings is denoted by $\mathfrak{F}$. 
\begin{definition}\label{definition 2.5}
\emph{\cite{n1} Let $\mathcal{G}([0,\infty))$ be the collection of all functions $\mathfrak{g}: [0,\infty) \rightarrow [0,\infty)$. Let $\mathcal{O}$ be the class of all operators $\mathfrak{O}(.,.):\mathcal{G}([0,\infty)) \rightarrow \mathcal{G}([0,\infty))$, $\mathfrak{g} \rightarrow \mathfrak{O}(\mathfrak{g},.)$ satisfying}

\emph{(i) $\mathfrak{O}(\mathfrak{g},0)=0$,}

\emph{(ii) $\mathfrak{O}(\mathfrak{g},\varrho)>0$ for $\varrho>0$,}

\emph{(iii) $\mathfrak{O}(\mathfrak{g},\varrho)\leq \mathfrak{O}(\mathfrak{g},\varsigma)$ for all $\varrho \leq \varsigma$,}

\emph{(iv) $\mathfrak{O}(\mathfrak{g},\max\{\varrho,\varsigma\})=\max\{\mathfrak{O}(\mathfrak{g},\varrho),\mathfrak{O}(\mathfrak{g},\varsigma)\}$ for all $\mathfrak{g} \in \mathcal{G}([0,\infty))$,}

\emph{(v) $\lim\limits_{n \rightarrow \infty}\mathfrak{O}(\mathfrak{g},\varrho_n)=\mathfrak{O}(\mathfrak{g},\lim\limits_{n \rightarrow \infty} \varrho_n)$.}
\end{definition}
\begin{definition}\label{definition 2.6}
\emph{Let $\mathfrak{X}$ be a Banach space equipped with a measure of noncompactness $\Xi$ containing a convex and nonempty pair $(\Delta,\Omega)$. A mapping $\daleth:\Delta \cup \Omega \rightarrow \Delta \cup \Omega$ is termed cyclic (noncyclic) operator type condensing mapping if $\daleth$ is a cyclic (noncyclic) mapping and for any  proximinal, nonempty, $\daleth$-invariant, convex, closed and bounded pair $(\omega,\Lambda)\subseteq (\Delta,\Omega)$ such that $dist(\omega,\Lambda)=dist(\Delta,\Omega)$ with  $\Xi(\daleth(\omega) \cup \daleth(\Lambda))>0$ and $\Xi(\omega \cup \Lambda)>0$, we have}
\begin{align*}
\theta(\varkappa(\mathfrak{O}(\mathfrak{g},\Xi(\daleth(\omega) \cup \daleth(\Lambda))),\mathfrak{O}(\mathfrak{g},\phi(\Xi(\daleth(\omega) \cup \daleth(\Lambda))))))& \leq \theta(\varkappa(\mathfrak{O}(\mathfrak{g},\Xi(\omega \cup \Lambda)),\mathfrak{O}(\mathfrak{g},\phi(\Xi(\omega \cup \Lambda)))))\\
&\quad - \psi(\theta(\varkappa(\mathfrak{O}(\mathfrak{g},\Xi(\omega \cup \Lambda)),\mathfrak{O}(\mathfrak{g},\phi(\Xi(\omega \cup \Lambda)))))),
\end{align*}
\emph{where $\theta \in \Theta$, $\phi \in \Phi$, $\psi \in \Psi$, $\varkappa \in \mathfrak{F}$ and $\mathfrak{O}(.,.) \in \mathcal{O}$.} 
\end{definition}
\begin{theorem}\label{theorem2.1}
Let $\mathfrak{X}$ be a Banach space equipped with a measure of noncompactness $\Xi$ containing a bounded, nonempty, convex and closed pair $(\Delta,\Omega)$ such that $\Delta_{0}$ is nonempty. Let $\daleth : \Delta \cup \Omega \to \Delta \cup \Omega$ be a cyclic relatively nonexpansive mapping that fulfills operator type  condensing mapping property according to Definition \ref{definition 2.6}. Then $\daleth$ possesses a best proximity point.
\end{theorem}
\begin{proof}
Since $\Delta_{0}$ is nonempty, $(\Delta_{0}, \Omega_{0})$ is nonempty. It can be easily verified that $(\Delta_{0}, \Omega_{0})$ is  proximinal, closed and convex. If $\varrho \in \Delta_{0}$, then there exists $\varsigma \in \Omega_{0}$ such that $\Vert \varrho-\varsigma \Vert = dist(\Delta,\Omega)$. Since $\daleth$ is a relatively nonexpansive mapping, $\Vert \daleth \varrho - \daleth \varsigma \Vert  \leq \Vert \varrho-\varsigma \Vert = dist(\Delta,\Omega)$ which implies that $\daleth \varrho \in \Omega_{0}$. Therefore, $\daleth(\Delta_{0}) \subseteq \Omega_{0}$. Similarly, we show that $\daleth(\Omega_{0}) \subseteq \Delta_{0}$ which gives $\daleth$ is cyclic on $\Delta_{0} \cup \Omega_{0}$. Let $\omega_{0} = \Delta_{0}$ and $\Lambda_{0}= \Omega_{0}$ and define
$$\omega_{n} = \overline{\operatorname{con}}(\daleth(\omega_{n-1})), \thinspace \thinspace
\Lambda_{n} = \overline{\operatorname{con}}(\daleth(\Lambda_{n-1})) \text{ for all } n \in \mathbb{N}.$$
Consider $\omega_1=\overline{\operatorname{con}}(\daleth (\omega_0))$. As $\daleth(\Delta_0) \subseteq \Omega_0$ and $\Delta_0=\omega_0$ and $\Omega_0=\Lambda_0$, therefore, $\daleth(\omega_0) \subseteq \Lambda_0$. This gives $\omega_1=\overline{\operatorname{con}}(\daleth (\omega_0)) \subseteq \overline{\operatorname{con}}(\Lambda_0)$. Also, $\Lambda_0$ is closed and convex, therefore, $\overline{\operatorname{con}}(\Lambda_0)=\Lambda_0$. Thus, 
$\omega_1=\overline{\operatorname{con}}(\daleth (\omega_0)) \subseteq \overline{\operatorname{con}}(\Lambda_0)=\Lambda_0$
which gives $\omega_1 \subseteq \Lambda_0$.  Similarly, $\omega_{2} = \overline{\operatorname{con}}(\daleth(\omega_{1})) \subseteq \overline{\operatorname{con}}(\daleth(\Lambda_{0})) = \Lambda_{1}$. Assume that the result is true for $p$ which gives $\omega_p \subseteq \Lambda_{p-1}$. Consider $\omega_{p+1}=\overline{\operatorname{con}}(\daleth (\omega_p))\subseteq \overline{\operatorname{con}}(\daleth (\Lambda_{p-1})) =\Lambda_p$ which implies that $\omega_{p+1} \subseteq \Lambda_p$. Hence, using induction, it follows that $\omega_{n+1} \subseteq \Lambda_n$ for all $n \in \mathbb{N} \cup \{0\}$. A similar analysis yields $\Lambda_{n} \subseteq \omega_{n-1}$ for all $n \in \mathbb{N}$. This establishes that  
$\omega_{n+2} \subseteq \Lambda_{n+1} \subseteq \omega_{n}$ and  $\Lambda_{n+2} \subseteq \omega_{n+1} \subseteq \Lambda_{n}$  for all  $n \in \mathbb{N} \cup \{0\}$. Therefore, $\{(\omega_{2n},\Lambda_{2n})\}_{n \in \mathbb{N} \cup \{0\}}$ is a decreasing sequence of convex, closed and  nonempty pairs in $\omega_{0} \times \Lambda_{0}$. Also,
\begin{align*}
\daleth(\omega_{2n}) \subseteq \daleth(\Lambda_{2n-1}) \subseteq \overline{\operatorname{con}}(\daleth(\Lambda_{2n-1})) = \Lambda_{2n}\\
\daleth(\Lambda_{2n}) \subseteq \daleth(\omega_{2n-1}) \subseteq \overline{\operatorname{con}}(\daleth(\omega_{2n-1})) = \omega_{2n}.
\end{align*}
Therefore, the pair $(\omega_{2n},\Lambda_{2n})$ is $\daleth$-invariant for all $n \in \mathbb{N}$. If $(\varrho,\varsigma) \in \Delta_{0} \times \Omega_{0}$ is a proximinal pair, then $\Vert \varrho-\varsigma \Vert=dist(\Delta,\Omega)$. Also,
$dist(\omega_{2n}, \Lambda_{2n}) \leq \Vert \daleth^{2n}\varrho - \daleth^{2n}\varsigma \Vert \leq \Vert \varrho - \varsigma \Vert = dist(\Delta,\Omega)$. Since $\omega_{2n} \subseteq \Delta$ and $\Lambda_{2n} \subseteq \Omega$, $dist(\omega_{2n},\Lambda_{2n}) \geq dist(\Delta,\Omega)$. This implies that $dist(\omega_{2n},\Lambda_{2n}) =dist(\Delta,\Omega)$ for all $n \in \mathbb{N} \cup \{0\}$.

We proceed to prove, using induction, that $(\omega_{n}, \Lambda_{n})$ is proximinal for all $n \in \mathbb{N} \cup \{0\}$. The result is trivially true for $n=0$. Suppose that $(\omega_k, \Lambda_{k})$ is proximinal. We show that $(\omega_{k+1}, \Lambda_{k+1})$ is proximinal.  For this, we first prove that $(\operatorname{con}(\daleth(\omega_{k})),\operatorname{con}(\daleth(\Lambda_{k})))$ is proximinal. Let $\vartheta \in \operatorname{con}(\daleth(\omega_{k}))$ be an arbitrary element. Then
\[
\vartheta = \sum_{j=1}^{m} \mu_{j} \daleth(\vartheta_{j})  \text{ with } \vartheta_{j} \in \omega_{k},\, m \in \mathbb{N},
\ \mu_{j} \geq 0 \text{ for all } 1\leq j \leq m  \text{ and } \sum_{j=1}^{m} \mu_{j} = 1.
\]
Owing to the proximinality of $(\omega_{k}, \Lambda_{k})$, there exists $\nu_{j} \in \Lambda_{k}$ $(1 \leq j \leq m)$ such that $\Vert \vartheta_{j} - \nu_{j} \Vert = dist(\omega_{k},\Lambda_{k})=dist(\Delta,\Omega)$. 
Let $\nu = \sum\limits_{j=1}^{m} \mu_{j} \daleth(\nu_{j})$. Then $\nu \in \operatorname{con}(\daleth(\Lambda_{k}))$. Consider
\begin{align*}
\Vert \vartheta - \nu\Vert
&= \Big\Vert \sum_{j=1}^{m} \mu_{j} \daleth(\vartheta_{j}) - \sum_{j=1}^{m} \mu_{j} \daleth(\nu_{j}) \Big\Vert \\
&\leq \sum_{j=1}^{m} \mu_{j} \Vert \daleth(\vartheta_{j})_-\daleth(\nu_{j})\Vert\\
& \leq \sum_{j=1}^{m} \mu_{j} \Vert \vartheta_{j}-\nu_{j})\Vert \\
&= \sum_{j=1}^{m} \mu_{j} dist(\omega_k,\Lambda_k)\\
&= dist(\Delta,\Omega).
\end{align*}
This gives $(\operatorname{con}(\daleth(\omega_{k})),\operatorname{con}(\daleth(\Lambda_{k})))$ is proximinal. Now, we show that $(\overline{\operatorname{con}}(\daleth(\omega_{k})),\overline{\operatorname{con}}(\daleth(\Lambda_{k})))$ is proximinal. Let $\zeta \in \overline{\operatorname{con}}(\daleth(\omega_{k}))$, then there exits a sequence $\{\zeta_n\}$ such that $\zeta_n \rightarrow \zeta$. Since $(\operatorname{con}(\daleth(\omega_{k})), \operatorname{con}(\daleth(\Lambda_{k})))$ is proximinal, there exists $\varpi_n \in \operatorname{con}(\daleth(\Lambda_{k}))$ such that $\Vert \zeta_n-\varpi_n \Vert =dist (\operatorname{con}(\\ \daleth(\omega_{k})),\operatorname{con}(\daleth(\Lambda_{k})))=dist (\Delta,\Omega)$. As $\overline{\operatorname{con}}(\daleth(\Lambda_{k}))$ is weakly compact, which guarantees the existence of a subsequence $\{\varpi_{n_k}\}$ of the sequence $\{\varpi_n\}$ converging a point $\varpi \in \overline{\operatorname{con}}(\daleth(\Lambda_{k}))$. We observe that 
$\Vert \zeta-\varpi \Vert \leq \liminf\limits_{k \rightarrow \infty} \Vert \zeta_{n_k}-\varpi_{n_k} \Vert=dist(\Delta,\Omega)$ which implies $\Vert \zeta-\varpi \Vert=dist(\Delta,\Omega)$ which gives $(\overline{\operatorname{con}}(\daleth(\omega_{k})), \overline{\operatorname{con}}(\daleth(\Lambda_{k})))$ is proximinal. Thus, we infer that $(\omega_{k+1},\Lambda_{k+1})$ is proximinal. 

Since $\omega_{2n+2} \subseteq \Lambda_{2n+1}$ and $\Lambda_{2n+2} \subseteq \omega_{2n+1}$, $\max\{\Xi(\omega_{2n+2}),\Xi(\Lambda_{2n+2})\} \leq \max\{\Xi(\Lambda_{2n+1}),\Xi(\omega_{2n+1})\}$. By Definition \ref{definition 2.5}(iii), we get $\mathfrak{O}(\mathfrak{g},\Xi(\omega_{2n+2} \cup \Lambda_{2n+2})) \leq \mathfrak{O}(\mathfrak{g}, \Xi(\omega_{2n+1} \cup \Lambda_{2n+1}))$. Using Definition \ref{definition 2.3} (i), we get $\phi(\Xi(\omega_{2n+2} \cup \Lambda_{2n+2})) \leq \phi(\Xi(\omega_{2n+1} \cup \Lambda_{2n+1}))$. Also, using Definition \ref{definition 2.5} (iii), we have $\mathfrak{O}(\mathfrak{g},\phi(\Xi(\omega_{2n+2} \cup \Lambda_{2n+2}))) \leq \mathfrak{O}(\mathfrak{g},\phi(\Xi(\omega_{2n+1} \cup \Lambda_{2n+1})))$. Since $F$ is nondecreasing, $\varkappa(\mathfrak{O}(\mathfrak{g},\Xi(\omega_{2n+2} \cup \Lambda_{2n+2})),\mathfrak{O}(\mathfrak{g},\phi(\Xi(\omega_{2n+2} \cup \Lambda_{2n+2})))) \leq \varkappa(\mathfrak{O}(\mathfrak{g}, \Xi(\omega_{2n+1} \cup \Lambda_{2n+1})),\mathfrak{O}(\mathfrak{g},\phi(\Xi(\omega_{2n+1} \cup \Lambda_{2n+1}))))$. Also, $\theta$ is a  strictly increasing function, we have 
\begin{align}\label{equation2.2}
\theta(\varkappa(\mathfrak{O}(\mathfrak{g},\Xi(\omega_{2n+2}& \cup \Lambda_{2n+2})),\mathfrak{O}(\mathfrak{g},\phi(\Xi(\omega_{2n+2} \cup \Lambda_{2n+2}))))) \nonumber\\
 &\leq \theta(\varkappa(\mathfrak{O}(\mathfrak{g}, \Xi(\omega_{2n+1} \cup \Lambda_{2n+1})),\mathfrak{O}(\mathfrak{g},\phi(\Xi(\omega_{2n+1} \cup \Lambda_{2n+1})))))\nonumber\\
&=\theta(\varkappa(\mathfrak{O}(\mathfrak{g}, \max\{\Xi(\omega_{2n+1}), \Xi(\Lambda_{2n+1})\}), \mathfrak{O}(\mathfrak{g},\phi(\max\{\Xi(\omega_{2n+1}), \Xi(\Lambda_{2n+1})\}))))\nonumber \\
&=\theta(\varkappa(\mathfrak{O}(\mathfrak{g}, \max\{\Xi(\overline{\operatorname{con}}(\daleth(\omega_{2n}))), \Xi(\overline{\operatorname{con}}(\daleth (\Lambda_{2n})))\}), \mathfrak{O}(\mathfrak{g},\phi(\max\{\overline{\operatorname{con}}(\Xi(\daleth(\omega_{2n}))),\nonumber \\
& \quad \Xi(\overline{\operatorname{con}}(\daleth(\Lambda_{2n})))\}))))\nonumber \\
&=\theta(\varkappa(\mathfrak{O}(\mathfrak{g}, \max\{\Xi(\daleth(\omega_{2n})), \Xi(\daleth (\Lambda_{2n}))\}), \mathfrak{O}(\mathfrak{g},\phi(\max\{\Xi(\daleth(\omega_{2n})),\Xi(\daleth(\Lambda_{2n}))\}))))\nonumber 
\end{align}
\begin{align}
&\hspace{2.2cm}=\theta(\varkappa(\mathfrak{O}(\mathfrak{g}, \Xi(\daleth(\omega_{2n})\cup \daleth (\Lambda_{2n}))), \mathfrak{O}(\mathfrak{g},\phi(\Xi(\daleth(\omega_{2n}))\cup \daleth(\Lambda_{2n}))))))\nonumber\\
&\hspace{2.2cm}\leq \theta(\varkappa(\mathfrak{O}(\mathfrak{g}, \Xi(\omega_{2n}\cup \Lambda_{2n})), \mathfrak{O}(\mathfrak{g},\phi(\Xi(\omega_{2n}\cup \Lambda_{2n})))))-\psi(\theta(\varkappa(\mathfrak{O}(\mathfrak{g}, \Xi(\omega_{2n}\cup \Lambda_{2n})), \nonumber \\
&\hspace{2.2cm} \quad \mathfrak{O}(\mathfrak{g},\phi(\Xi(\omega_{2n}\cup \Lambda_{2n}))))))
\end{align}
As $\{\Xi(\omega_{2n} \cup \Lambda_{2n})\}$ is a sequence of positive real numbers which is bounded below and decreasing, we deduce  that $\{\Xi(\omega_{2n} \cup \Lambda_{2n})\}_{n \in \mathbb{N} \cup \{0\}}$ is convergent. Suppose that $\lim\limits_{n \rightarrow \infty}\Xi(\omega_{2n} \cup \Lambda_{2n})=\xi$. We show that $\xi=0$. If $\xi>0$, letting $n \to \infty$ and taking the upper limit in (\ref{equation2.2}), we obtain
\begin{align*}
\limsup\limits_{n \rightarrow \infty} \theta(\varkappa(\mathfrak{O}(\mathfrak{g},\Xi(\omega_{2n+2}& \cup \Lambda_{2n+2})),\mathfrak{O}(\mathfrak{g},\phi(\Xi(\omega_{2n+2} \cup \Lambda_{2n+2})))))\\& \leq \limsup\limits_{n \rightarrow \infty}\theta(\varkappa(\mathfrak{O}(\mathfrak{g}, \Xi(\omega_{2n}\cup \Lambda_{2n})), \mathfrak{O}(\mathfrak{g},\phi(\Xi(\omega_{2n}\cup \Lambda_{2n})))))\\
& \quad - \liminf\limits_{n \rightarrow \infty}\psi(\theta(\varkappa(\mathfrak{O}(\mathfrak{g}, \Xi(\omega_{2n}\cup \Lambda_{2n})), \mathfrak{O}(\mathfrak{g},\phi(\Xi(\omega_{2n}\cup \Lambda_{2n}))))))
\end{align*}
which gives $\theta(\varkappa(\mathfrak{O}(\mathfrak{g},\xi),\mathfrak{O}(\mathfrak{g},\phi(\xi)))) \leq \theta(\varkappa(\mathfrak{O}(\mathfrak{g},\xi),\mathfrak{O}(\mathfrak{g},\phi(\xi))))-\psi(\theta(\varkappa(\mathfrak{O}(\mathfrak{g},\xi),\mathfrak{O}(\mathfrak{g},\phi(\xi)))))$. This implies that $\psi(\theta(\varkappa(\mathfrak{O}(\mathfrak{g},\xi),\mathfrak{O}(\mathfrak{g},\phi(\xi)))))=0$. Using Definition \ref{definition 2.2}(iii), we deduce that $\theta(\varkappa(\mathfrak{O}(\mathfrak{g},\xi),\\ \mathfrak{O}(\mathfrak{g},\phi(\xi))))=1$. By using Definition \ref{definition 2.1}(ii), we infer that $\varkappa(\mathfrak{O}(\mathfrak{g},\xi),\mathfrak{O}(\mathfrak{g},\phi(\xi)))=0$. Also, $\max\{\mathfrak{O}(\mathfrak{g},\xi),\mathfrak{O}(\mathfrak{g},\phi(\xi))\}\leq \varkappa(\mathfrak{O}(\mathfrak{g},\xi),\mathfrak{O}(\mathfrak{g},\phi(\xi)))$, which implies that $\max\{\mathfrak{O}(\mathfrak{g},\xi),\mathfrak{O}(\mathfrak{g},\phi(\xi))\}=0$. Therefore, $\mathfrak{O}(\mathfrak{g},\xi)=0$ and $\mathfrak{O}(\mathfrak{g},\phi(\xi))=0$ which gives $\xi=0$. Hence, $\lim\limits_{n \to \infty} \Xi(\omega_{2n} \cup \Lambda_{2n})=0$ which gives $\lim\limits_{n \to \infty} \Xi(\omega_{2n})=0$ and $\lim\limits_{n \to \infty} \Xi(\Lambda_{2n})=0$. Let $\omega_{\infty} = \bigcap\limits_{n=0}^{\infty} \omega_{2n}$ and $\Lambda_{\infty} = \bigcap\limits_{n=0}^{\infty} \Lambda_{2n}$. By property (iv) of measure of noncompactness, $(\omega_{\infty}, \Lambda_{\infty})$ is nonempty and compact. Also, $(\omega_{\infty}, \Lambda_{\infty})$ is convex and $\daleth$-invariant with $dist(\omega_{\infty}, \Lambda_{\infty})=dist(\Delta,\Omega)$. Hence, using Theorem \ref{theorem1}, we infer that $\daleth$ possesses a best proximity point. 
\end{proof}
Substituting $\theta(\varrho)=e^{\varrho}$, $\mathfrak{g} \equiv I$,   $\mathfrak{O}(\mathfrak{g},u)=u$, $\phi(\varrho)=\varrho$, $\varkappa(\varrho,\varsigma)=\varrho+\varsigma$ and $\psi(\varrho)=\varrho-\varrho^\alpha$, where $\alpha \in [0,1)$ into the above theorem yields the following result:
\begin{corollary}\label{corollary1}
Let $\mathfrak{X}$ be a Banach space equipped with a measure of noncompactness $\Xi$ containing  a  bounded, nonempty, convex and closed pair $(\Delta,\Omega)$ such that $\Delta_{0}$ is nonempty. Let $\daleth : \Delta \cup \Omega \to \Delta \cup \Omega$ be a cyclic relatively nonexpansive mapping such that for any  proximinal, nonempty, $\daleth$-invariant, convex, bounded and closed pair $(\omega,\Lambda)\subseteq (\Delta,\Omega)$ such that $dist(\omega,\Lambda)=dist(\Delta,\Omega)$ with  $\Xi(\daleth(\omega) \cup \daleth(\Lambda))>0$ and $\Xi(\omega \cup \Lambda)>0$, we have
$$\Xi(\daleth(\omega) \cup \daleth(\Lambda)) \leq \alpha \Xi(\omega \cup \Lambda),$$
where $\alpha \in [0,1)$. Then $\daleth$ possesses a best proximity point.
\end{corollary}
Substituting $\mathfrak{g}\equiv I$, $\mathfrak{O}(\mathfrak{g},u)=u$ and $\varkappa(\varrho,\varsigma)=\varrho+\varsigma$ into Theorem \ref{theorem2.1} yields the following result:
\begin{corollary}
Let $\mathfrak{X}$ be a Banach space equipped with a measure of noncompactness $\Xi$ containing  a  bounded, nonempty, convex and closed pair $(\Delta,\Omega)$ such that $\Delta_{0}$ is nonempty. Let $\daleth: \Delta \cup \Omega \to \Delta \cup \Omega$ be a cyclic relatively nonexpansive mapping such that for any  proximinal, nonempty, $\daleth$-invariant, convex, bounded and closed pair $(\omega,\Lambda)\subseteq (\Delta,\Omega)$ such that $dist(\omega,\Lambda)=dist(\Delta,\Omega)$ with  $\Xi(\daleth(\omega) \cup \daleth(\Lambda))>0$ and $\Xi(\omega \cup \Lambda)>0$, we have
$$\theta(\Xi(\daleth(\omega) \cup \daleth(\Lambda))+\phi(\Xi(\daleth(\omega) \cup \daleth(\Lambda)))) \leq \theta(\Xi(\omega \cup \Lambda)+\phi(\Xi(\omega \cup \Lambda)))-\psi(\theta(\Xi(\omega \cup \Lambda)+\phi(\Xi(\omega \cup \Lambda)))),$$
where $\theta \in \Theta$, $\phi \in \Phi$. Then $\daleth$ possesses a best proximity point.
\end{corollary}
We now proceed to establish similar result for noncyclic condensing mappings.
\begin{theorem}\label{theorem2.2}
Let $\mathfrak{X}$ be a strictly convex Banach space equipped with a measure of noncompactness $\Xi$ containing a  bounded, nonempty, convex and closed pair $(\Delta,\Omega)$  such that $\Delta_{0}$ is nonempty. Let $\daleth : \Delta \cup \Omega \to \Delta \cup \Omega$ be a noncyclic relatively nonexpansive mapping that fulfills operator type  condensing mapping property according to Definition \ref{definition 2.6}. Then $\daleth$ possesses a best proximity pair.
\end{theorem}
\begin{proof}
As done in Theorem \ref{theorem2.1}, define $\omega_{n}=\overline{\operatorname{con}}(\daleth(\omega_{n-1}))$ and
$\Lambda_{n}=\overline{\operatorname{con}}(\daleth(\Lambda_{n-1}))$ for all $n\in\mathbb{N}$. Adopting the technique used in the proof of Theorem \ref{theorem2.1}, we establish that $\{(\omega_{n},\Lambda_{n})\}$ constitutes a decreasing sequence of  proximinal, nonempty, closed, $\daleth$-invariant and convex pairs such that $dist(\omega_{n},\Lambda_{n})=dist(\Delta,\Omega)$ for all $n\in\mathbb{N}\cup\{0\}$. Consider 
\begin{align}\label{equation2.3}
\theta(\varkappa(\mathfrak{O}(\mathfrak{g}, \Xi(&\omega_{n+1} \cup \Lambda_{n+1})),\mathfrak{O}(\mathfrak{g},\phi(\Xi(\omega_{n+1} \cup \Lambda_{n+1})))))\nonumber\\
&=\theta(\varkappa(\mathfrak{O}(\mathfrak{g}, \max\{\Xi(\omega_{n+1}), \Xi(\Lambda_{n+1}\}), \mathfrak{O}(\mathfrak{g},\phi(\max\{\Xi(\omega_{n+1}), \Xi(\Lambda_{n+1})\}))))\nonumber \\
&=\theta(\varkappa(\mathfrak{O}(\mathfrak{g}, \max\{\Xi(\overline{\operatorname{con}}(\daleth(\omega_{n}))), \Xi(\overline{\operatorname{con}}(\daleth (\Lambda_{n})))\}), \mathfrak{O}(\mathfrak{g},\phi(\max\{\overline{\operatorname{con}}(\Xi(\daleth(\omega_{n}))),\nonumber \\
& \quad \Xi(\overline{\operatorname{con}}(\daleth(\Lambda_{n})))\}))))\nonumber \\
&=\theta(\varkappa(\mathfrak{O}(\mathfrak{g}, \max\{\Xi(\daleth(\omega_{n})), \Xi(\daleth (\Lambda_{n}))\}), \mathfrak{O}(\mathfrak{g},\phi(\max\{\Xi(\daleth(\omega_{n})),\Xi(\daleth(\Lambda_{n}))\}))))\nonumber\\
&=\theta(\varkappa(\mathfrak{O}(\mathfrak{g}, \Xi(\daleth(\omega_{n})\cup \daleth (\Lambda_{n}))), \mathfrak{O}(\mathfrak{g},\phi(\Xi(\daleth(\omega_{n}))\cup \daleth(\Lambda_{n})))))\nonumber\\
&\leq \theta(\varkappa(\mathfrak{O}(\mathfrak{g}, \Xi(\omega_{n}\cup \Lambda_{n})), \mathfrak{O}(\mathfrak{g},\phi(\Xi(\omega_{n}\cup \Lambda_{n})))))-\psi(\theta(\varkappa(\mathfrak{O}(\mathfrak{g}, \Xi(\omega_{n}\cup \Lambda_{n})), \nonumber \\
& \quad \mathfrak{O}(\mathfrak{g},\phi(\Xi(\omega_{n}\cup \Lambda_{n}))))))
\end{align}
As $\{\Xi(\omega_{n} \cup \Lambda_{n})\}$ is a sequence of positive real numbers which is bounded below and decreasing, we deduce  that $\{\phi(\Upsilon(\omega_{n} \cup \Lambda_{n}))\}_{n \in \mathbb{N} \cup \{0\}}$ is convergent. Suppose that $\lim\limits_{n \rightarrow \infty}\Xi(\omega_{n} \cup \Lambda_{n})=\upsilon$. We show that $\upsilon=0$. If $\upsilon>0$,  taking upper limit in (\ref{equation2.3}) as $n \rightarrow \infty$, we get
\begin{align*}
\limsup\limits_{n \rightarrow \infty} \theta(\varkappa(\mathfrak{O}(\mathfrak{g},\Xi(\omega_{n+1}& \cup \Lambda_{n+1})),\mathfrak{O}(\mathfrak{g},\phi(\Xi(\omega_{n+1} \cup \Lambda_{n+1})))))\\& \leq \limsup\limits_{n \rightarrow \infty}\theta(\varkappa(\mathfrak{O}(\mathfrak{g}, \Xi(\omega_{n}\cup \Lambda_{n})), \mathfrak{O}(\mathfrak{g},\phi(\Xi(\omega_{n}\cup \Lambda_{n})))))\\
& \quad - \liminf\limits_{n \rightarrow \infty}\psi(\theta(\varkappa(\mathfrak{O}(\mathfrak{g}, \Xi(\omega_{n}\cup \Lambda_{n})), \mathfrak{O}(\mathfrak{g},\phi(\Xi(\omega_{n}\cup \Lambda_{n}))))))
\end{align*}
which gives $\theta(\varkappa(\mathfrak{O}(\mathfrak{g},\upsilon),\mathfrak{O}(\mathfrak{g},\phi(\upsilon)))) \leq \theta(\varkappa(\mathfrak{O}(\mathfrak{g},\upsilon),\mathfrak{O}(\mathfrak{g},\phi(\upsilon))))-\psi(\theta(\varkappa(\mathfrak{O}(\mathfrak{g},\upsilon),\mathfrak{O}(\mathfrak{g},\phi(\upsilon)))))$. This implies that $\psi(\theta(\varkappa(\mathfrak{O}(\mathfrak{g},\upsilon),\mathfrak{O}(\mathfrak{g},\phi(\upsilon)))))=0$. Using Definition \ref{definition 2.2}(iii), we deduce that $\theta(\varkappa(\mathfrak{O}(\mathfrak{g},\\ \upsilon),\mathfrak{O}(\mathfrak{g},\phi(\upsilon))))=1$. By using Definition \ref{definition 2.1}(ii), we infer that $\varkappa(\mathfrak{O}(\mathfrak{g},\upsilon),\mathfrak{O}(\mathfrak{g},\phi(\upsilon)))=0$. Also, $\max\{\mathfrak{O}(\mathfrak{g},\upsilon),\mathfrak{O}(\mathfrak{g},\phi(\upsilon))\}\leq \varkappa(\mathfrak{O}(\mathfrak{g},\upsilon),\mathfrak{O}(\mathfrak{g},\phi(\upsilon)))$, which implies that $\max\{\mathfrak{O}(\mathfrak{g},\upsilon),\mathfrak{O}(\mathfrak{g},\phi(\upsilon))\}=0$. Therefore, $\mathfrak{O}(\mathfrak{g},\upsilon)=0$ and $\mathfrak{O}(\mathfrak{g},\phi(\upsilon))=0$ which gives $\upsilon=0$. Hence, $\lim\limits_{n \to \infty} \Xi(\omega_{n} \cup \Lambda_{n})=0$ which gives $\lim\limits_{n \to \infty} \Xi(\omega_{n})=0$ and $\lim\limits_{n \to \infty} \Xi(\Lambda_{n})=0$. Let $\omega_{\infty} = \bigcap\limits_{n=0}^{\infty} \omega_{n}$ and $\Lambda_{\infty} = \bigcap\limits_{n=0}^{\infty} \Lambda_{n}$. By property (iv) of measure of noncompactness, $(\omega_{\infty}, \Lambda_{\infty})$ is nonempty and compact. We also see that $(\omega_{\infty}, \Lambda_{\infty})$ is  $\daleth$-invariant and convex such that  $dist(\omega_{\infty}, \Lambda_{\infty})=dist(\Delta,\Omega)$. Consequently, using Theorem \ref{theorem2}, we establish that $\daleth$ possesses a best proximity pair. 
\end{proof}
\section{Application}
Infinite system of  differential equations feature prominently in several areas of nonlinear analysis. For instance, infinite system of differential equations often emerge when methods based on the numerical analysis are employed to solve partial differential equations. A wide range of real-world applications can be addressed through the concept of infinite system of differential equations. This section is dedicated to the analysis of an infinite system of second order differential equations subject to boundary conditions, with the purpose of ensuring the existence of its  optimal solution under appropriate assumptions. 
 
Let $(\mathfrak{X},\Vert . \Vert_{\mathfrak{X}})$ be a Banach space.  We write $ \mathcal{C}(I,\mathfrak{X})$ for the Banach space consisting of all continuous functions from the interval $I=[0,\tau]$ to $\mathfrak{X}$ equipped with the supremum norm $\Vert x \Vert=\sup \limits_{\varpi\in I} \Vert x(\varpi) \Vert_{\mathfrak{X}}$. We denote by $\mathcal{C}(I,\mathbb{R})$ the space of all continuous real functions defined on the interval $I$ and $\mathcal{C}^2(I,\mathbb{R})$  represents the class of functions possessing continuous second derivatives on $I$. Consider the following infinite system of second order differential equations  
\begin{equation}\label{nequation3.1}
\begin{aligned}
\frac{d^2\hbar_i}{d\varpi^2}&=-\varphi_i(\varpi,\hbar_1,\hbar_2,\hbar_3,\ldots), \quad \hbar_i(0)=\hbar_i(\tau)=\rho_i, \quad t \in [0,\tau],
\end{aligned}
\end{equation}
where $\rho_i \in \mathbb{R}$ and $\varphi_i(\varpi,\hbar) \in \mathcal{C}(I,\mathbb{R})$.  Then $\hbar \in \mathcal{C}^2(I,\mathbb{R})$ is a solution of (\ref{nequation3.1}) if and only if $\hbar \in \mathcal{C}(I,\mathbb{R})$ is a solution of the infinite system of integral equations 
\begin{equation}
\hbar_i(\varpi)=\rho_i+\int_0^\tau \wp(\varpi,\varrho)\varphi_i(\varrho,\hbar(\varrho))d\varrho, \nonumber
\end{equation}
where $\wp(\varpi,\varrho)=\left\{\begin{array}{ll} \frac{\varpi}{\tau}(\tau-\varrho), & 0\leq \varpi \leq \varrho \leq \tau,\\
\frac{\varrho}{\tau}(\tau-\varpi), & 0\leq \varrho \leq \varpi \leq \tau \end{array} \right.$ is the Green's function associated with (\ref{nequation3.1}).

 Throughout this section, we examine the following infinite system of second order differential equations  
\begin{equation}\label{equation3.1}
\begin{aligned}
\frac{d^2\hbar_i}{d\varpi^2}&=-\varphi_i(\varpi,\hbar_1,\hbar_2,\hbar_3,\ldots), \quad \hbar_i(0)=\hbar_i(\tau)=\rho_i, \quad t \in [0,\tau],\\
\frac{d^2\ell_i}{d\varpi^2}&=-\chi_i(\varpi,\ell_1,\ell_2,\ell_3,\ldots), \quad \ell_i(0)=\ell_i(\tau)=\lambda_i, \quad t \in [0,\tau],
\end{aligned}
\end{equation}
where $\rho_i, \lambda_i \in \mathbb{R}$ and $\varphi_i(\varpi,\hbar), \chi_i(\varpi,\ell) \in \mathcal{C}(I,\mathbb{R})$. We investigate the system (\ref{equation3.1}) in Banach sequence space $c_0$ and $l_1$. The following results are instrumental in proving the results of this section:
\begin{theorem}\label{thm1}
\cite{n5} Let $\mathfrak{X}$ be a Banach space, $I=[a,b]$ and $\mathfrak{K} \subseteq \mathcal{C}(I,\mathfrak{X})$ be a bounded and equicontinuous set. Then $\Upsilon(\mathfrak{K})=\sup\limits_{\varpi \in I}\Upsilon(\{x(\varpi):x \in \mathfrak{K}\}).$
\end{theorem} 
Recall that in the sequence space $(c_0,\Vert . \Vert_{c_0})$, for a bounded subset $\mathcal{B}$ of $c_0$, the Hausdorff measure of noncompactness $\Upsilon$  is given by  
\begin{align}\label{usec0}
\Upsilon(\mathcal{B})=\lim_{n\rightarrow \infty} \Big\{\sup_{(\omega_i)\in \mathcal{B}} \Big(\max_{i\geq n}\vert \omega_i\vert \Big)\Big\}.
\end{align}
Also, in the sequence space $(l_1,\Vert . \Vert_{l_1})$, for a bounded subset $\mathcal{B}$ of $c_0$, the Hausdorff measure of noncompactness $\Upsilon$ is given by
\begin{align}\label{usel1}
\Upsilon(\mathcal{B})=\lim_{n\rightarrow \infty} \Big\{\sup_{(\omega_i)\in \mathcal{B}} \Big(\sum_{i=n}^\infty \vert \omega_i\vert \Big)\Big\}.
\end{align}
\subsection{Solution in $c_0$ space}
Let $\rho=(\rho_1,\rho_2,\rho_3,\ldots), \lambda=(\lambda_1,\lambda_2,\lambda_3,\ldots) \in c_0$. Let $\mathfrak{U_1}=\overline{\mathcal{B}}(\rho,r) \subseteq c_0$ and $\mathfrak{U_2}=\overline{\mathcal{B}}(\lambda,r) \subseteq c_0$. Let
$$\mathfrak{M_1}=\mathcal{C}(I,\mathfrak{U_1})=\{\hbar \in \mathcal{C}(I,c_0) \mbox{ such that }\hbar:I \rightarrow \mathfrak{U_1}, \hbar_i(0)=\hbar_i(\tau)=\rho_i\},$$
$$\mathfrak{M_1}=\mathcal{C}(I,\mathfrak{U_2})=\{\hbar \in \mathcal{C}(I,c_0) \mbox{ such that }\hbar:I \rightarrow \mathfrak{U_2}, \hbar_i(0)=\hbar_i(\tau)=\lambda_i\}.$$ Then $(\mathfrak{M_1},\mathfrak{M_2})$ is a  bounded, nonempty, convex and  closed pair in $\mathcal{C}(I,c_0)$ with $dist(\mathfrak{M_1},\mathfrak{M_2})=\Vert \rho-\lambda \Vert_{c_0}$.
Define $\aleph:\mathfrak{M_1} \cup \mathfrak{M_2} \rightarrow \mathfrak{M_1} \cup \mathfrak{M_2}$ by 
\begin{equation}\label{equation3.2}
\begin{aligned}
(\aleph \hbar)(\varpi)=\{(\aleph \hbar)_i(\varpi)\}=\left\{\left\{\begin{array}{ll} \lambda_i+\int_0^\tau \wp(\varpi,\varrho)\varphi_i(\varrho,\hbar(\varrho))d\varrho, & \hbar \in \mathfrak{M_1},\\
\rho_i+\int_0^\tau \wp(\varpi,\varrho)\chi_i(\varrho,\hbar(\varrho))d\varrho, & \hbar \in \mathfrak{M_2}. \end{array} \right\}\right., \quad \varpi \in I
\end{aligned}
\end{equation}
We say that $\upsilon \in \mathfrak{M_1} \cup \mathfrak{M_2}$ is an optimum solution of (\ref{equation3.1}) if the mapping given by (\ref{equation3.2}) possess a best proximity point $\upsilon$, that is, $\Vert \aleph \upsilon-\upsilon \Vert=dist(\mathfrak{M_1},\mathfrak{M_2})$. 

We establish our results under the following assumptions: 

($\mathfrak{A}_1$) The functions $\varphi(\varpi,\hbar(\varpi))=(\varphi_1(\varpi,\hbar(\varpi)),\varphi_2(\varpi,\hbar(\varpi)),\varphi_3(\varpi,\hbar(\varpi)),\ldots)$ and $\chi(\varpi,\hbar(\varpi))=(\chi_1(\varpi, \hbar(\varpi)), \chi_2(\varpi,\hbar(\varpi)),\chi_3(\varpi,\hbar(\varpi)),\ldots)$ are defined on $I \times c_0$ to $c_0$ and each $\varphi_i$ and $\chi_i$ are real valued functions such that 
$$\Vert \varphi(\varpi,\hbar(\varpi))-\chi(\varpi,\ell(\varpi)) \Vert_{c_0} \leq \frac{8}{\tau^2}[ \Vert \hbar(\varpi)-\ell(\varpi) \Vert_{c_0}-\Vert \lambda-\rho \Vert_{c_0}].$$

($\mathfrak{A}_2$) $\varphi(.,\hbar)$ and $\chi(.,\hbar)$ are measurable for each fixed $\hbar$.

($\mathfrak{A}_3$)  For each $\varpi \in I$, $\hbar(\varpi)=(\hbar_i(\varpi)) \in c_0$ and for each $i \in \mathbb{N}$, we have
\begin{align*}
\vert \varphi_i(\varpi,\hbar(\varpi))\vert &\leq p_i(\varpi)+q_i(\varpi)\sup \{\vert \hbar_j\vert:j\geq i\},\\
\vert \chi_i(\varpi,\hbar(\varpi))\vert &\leq k_i(\varpi)+l_i(\varpi)\sup \{\vert \hbar_j\vert:j\geq i\},
\end{align*}
where $p_i(\varpi), q_i(\varpi), k_i(\varpi), l_i(\varpi) \in \mathcal{C}(I,\mathbb{R})$ such that the sequences $(p_i(\varpi))$ and $(k_i(\varpi))$ converges uniformly on $I$ to a identically zero function on $I$ and the sequences $(q_i(\varpi))$ and $(l_i(\varpi))$ are equibounded on $I$.

We adopt the following notation:
\begin{align*}
P&=\sup \{p_i(\varpi):\varpi \in I, i \in \mathbb{N}\},\\
q(t)&=\sup \{q_i(\varpi):i \in \mathbb{N}\},\\
Q&=\sup \{q(\varpi):\varpi \in I\},\\
K&=\sup \{k_i(\varpi):\varpi \in I, i \in \mathbb{N}\},\\
l(t)&=\sup \{l_i(\varpi):i \in \mathbb{N}\},\\
L&=\sup \{l(\varpi):\varpi \in I\}.
\end{align*}
\begin{theorem}
Along with the assumptions $(\mathfrak{A}_1)$-$(\mathfrak{A}_3)$, $\frac{Q\tau^2}{8}<1$ and $\frac{L\tau^2}{8}<1$, the infinite system of second order differential equations (\ref{equation3.1}) has an optimum solution.
\end{theorem}
\begin{proof}
Consider
\begin{align*}
\Vert (\aleph \hbar)(\varpi)-\lambda \Vert _{c_0}&=\sup_{i\geq 1}\vert (\aleph  \hbar)_i(\varpi)-\lambda_i \vert \\
&=\sup_{i \geq 1} \Big\vert \lambda_i+\int_0^\tau \wp(\varpi,\varrho)\varphi_i(\varrho,\hbar(\varrho))d\varrho-\lambda_i \Big\vert\\
& \leq \sup_{i \geq 1} \int_0^\tau \vert \wp(\varpi,\varrho) \vert \vert \varphi_i(\varrho,\hbar(\varrho))\vert d\varrho \\
& \leq \sup_{i \geq 1} \int_0^\tau \vert \wp(\varpi,\varrho) \vert [ p_i(\varrho)+q_i(\varrho)\sup \{\vert \hbar_j(\varrho)\vert:j \geq i\} ]d\varrho\\
& \leq \frac{\tau^2}{8}[P+Q\Vert \hbar(\varpi) \Vert _{c_0}]. 
\end{align*}
As $\Vert \hbar(\varpi) \Vert _{c_0}\leq  \Vert \hbar(\varpi)-\rho \Vert_{c_0}+\Vert \rho \Vert _{c_0}$, This gives 
$\Vert (\aleph \hbar)(\varpi)-\lambda \Vert _{c_0} < \frac{\tau^2}{8}[P_1+Q\Vert \hbar(\varpi) -\rho\Vert_{c_0}]$, where $P_1=P+Q\Vert \rho \Vert_{c_0}$. Then $r_1=\frac{\tau^2P_1}{8-\tau^2Q}$ is the optimal solution of this inequality
$$\frac{\tau^2}{8}[P_1+Qr] \leq r.$$
Consider
\begin{align*}
\Vert (\aleph \ell)(\varpi)-\rho \Vert _{c_0}&=\sup_{i \geq 1} \vert (\aleph \ell)_i(\varpi)-\rho_i\vert \\
&=\sup_{i \geq 1} \Big\vert \rho_i+\int_0^\tau \wp(\varpi,\varrho)\chi_i(\varrho,\ell(\varrho))d\varrho-\rho_i \Big\vert\\
& \leq \sup_{i \geq 1} \int_0^\tau \vert \wp(\varpi,\varrho) \vert \vert \chi_i(\varrho,\ell(\varrho))\vert d\varrho \\
& \leq \sup_{i \geq 1} \int_0^\tau \vert \wp(\varpi,\varrho) \vert [ k_i(\varrho)+l_i(\varrho)\sup \{\vert \ell_j(\varrho) \vert:j \geq i\} ]d\varrho\\
& \leq \frac{\tau^2}{8}[K+L\Vert \ell(\varpi) \Vert_{c_0}].
\end{align*}
As $\Vert \ell(\varpi) \Vert _{c_0}\leq  \Vert \ell(\varpi)-\lambda \Vert_{c_0}+\Vert \lambda \Vert _{c_0}$, This gives 
$\Vert (\aleph \ell)(\varpi)-\rho \Vert_{c_0} < \frac{\tau^2}{8}[K_1+L\Vert \ell -\lambda \Vert_{c_0}]$, where $K_1=K+L\Vert \lambda \Vert_{c_0}$. Then $r_2=\frac{\tau^2K_1}{8-\tau^2L}$ is the optimal solution of this inequality
$$\frac{\tau^2}{8}[K_1+Lr] \leq r.$$
Let $r=\max\{r_1,r_2\}$. Now, we show that $\aleph$ is a cyclic mapping on $\mathfrak{M_1} \cup \mathfrak{M_2}$.  First, we prove that $\aleph(\mathfrak{M_1}) \subseteq \mathfrak{M_2}$. Let $\hbar \in \mathfrak{M_1}$. Consider 
\begin{align*}
(\aleph \hbar)_i(0)&=\lambda_i+\int_0^\tau \wp(0,\varrho)\varphi_i(\varrho,\hbar(\varrho))d\varrho\\
&=\lambda_i+\int_0^\tau 0.\varphi_i(\varrho,\hbar(\varrho))d\varrho\\
&=\lambda_i
\end{align*} 
and 
\begin{align*}
(\aleph \hbar)_i(\tau)&=\lambda_i+\int_0^\tau \wp(\tau,\varrho)\varphi_i(\varrho,\hbar(\varrho))d\varrho\\
&=\lambda_i+\int_0^\tau 0.\varphi_i(\varrho,\hbar(\varrho))d\varrho\\
&=\lambda_i.
\end{align*}
Also, $\Vert (\aleph \hbar)(\varpi)-\lambda \Vert _{c_0}\leq r$. This implies that $\aleph \hbar \in \mathfrak{M_2}$. Now, we show that $\aleph(\mathfrak{M_2}) \subseteq \mathfrak{M_1}$. Let $\ell \in \mathfrak{M_2}$. Consider 
\begin{align*}
(\aleph \ell)_i(0)&=\rho_i+\int_0^\tau \wp(0,\varrho)\chi_i(\varrho,\ell(\varrho))d\varrho\\
&=\rho_i+\int_0^\tau 0.\chi_i(\varrho,\ell(\varrho))d\varrho\\
&=\rho_i
\end{align*} 
and 
\begin{align*}
(\aleph \ell)_i(\tau)&=\rho_i+\int_0^\tau \wp(\tau,\varrho)\chi_i(\varrho,\ell(\varrho))d\varrho\\
&=\rho_i+\int_0^\tau 0.\chi_i(\varrho,\ell(\varrho))d\varrho\\
&=\rho_i.
\end{align*}
Also, $\Vert (\aleph \ell)(\varpi)-\rho \Vert _{c_0}\leq r$. This implies that $\aleph \ell \in \mathfrak{M_1}$. Therefore, $\aleph$ is a cyclic mapping. Clearly, $\aleph(\mathfrak{M_1})$ and $\aleph(\mathfrak{M_2})$ are  bounded subsets of $\mathcal{C}(I,c_0)$. We proceed to establish that $\aleph$ is a relatively nonexpansive mapping. Let $\hbar \in \mathfrak{M_1}$ and $\ell \in \mathfrak{M_2}$. Consider 
\begin{align*}
\Vert (\aleph \hbar)(\varpi)-(\aleph \ell)(\varpi) \Vert_{c_0}&=\sup_{i \geq 1} \Big\vert \lambda_i+\int_0^\tau \wp(\varpi,\varrho)\varphi_i(\varrho,\hbar(\varrho))d\varrho-\rho_i-\int_0^\tau \wp(\varpi,\varrho)\chi_i(\varrho,\ell(\varrho))d\varrho \Big\vert \\
& \leq \sup_{i \geq 1} \Big[\vert \lambda_i-\rho_i \vert +\int_0^\tau \vert \wp(\varpi,\varrho)\vert \vert \varphi_i(\varrho,\hbar(\varrho))-\chi_i(\varrho,\ell(\varrho))\vert d\varrho \Big]\\
& \leq \sup_{i \geq 1}\vert \lambda_i-\rho_i \vert + \sup_{i \geq 1} \int_0^\tau \vert \wp(\varpi,\varrho)\vert \vert \varphi_i(\varrho,\hbar(\varrho))-\chi_i(\varrho,\ell(\varrho))\vert d\varrho \\
& \leq \sup_{i \geq 1}\vert \lambda_i-\rho_i \vert +  \int_0^\tau \vert \wp(\varpi,\varrho)\vert \sup_{i \geq 1}\vert \varphi_i(\varrho,\hbar(\varrho))-\chi_i(\varrho,\ell(\varrho))\vert d\varrho \\
& = \sup_{i \geq 1}\vert \lambda_i-\rho_i \vert +  \int_0^\tau \vert \wp(\varpi,s)\vert \Vert \varphi(\varrho,\hbar(\varrho))-\chi(\varrho,\ell(\varrho))\Vert_{c_0} d\varrho \\
&\leq \Vert \lambda-\rho\Vert_{c_0}+\frac{8}{\tau^2}\int_0^\tau \vert \wp(\varpi,\varrho)\vert [\Vert \hbar-\ell \Vert_{\mathcal{C}(I,c_0)}-\Vert \lambda-\rho\Vert_{c_0}]ds\\
&\leq \Vert \hbar-\ell \Vert_{\mathcal{C}(I,c_0)}
\end{align*}
which implies that $\Vert \aleph \hbar-\aleph \ell \Vert_{\mathcal{C}(I,c_0)} \leq \Vert \hbar-\ell \Vert_{\mathcal{C}(I,c_0)}$. 
Now, we show that $\aleph(\mathfrak{M_1})$ is an equicontinuous subset of $\mathfrak{M_2}$. Consider 
\begin{align*}
\Vert (\aleph \hbar)(\varpi)-(\aleph \hbar)(\varpi')\Vert_{c_0}&=\sup_{i \geq 1} \Big\vert \lambda_i+\int_0^\tau \wp(\varpi,\varrho)\varphi_i(\varrho,\hbar(\varrho))d\varrho-\lambda_i-\int_0^\tau \wp(\varpi',\varrho)\varphi_i(\varrho,\hbar(\varrho))d\varrho \Big\vert\\
& \leq \sup_{i \geq 1}\int_0^\tau \vert \wp(\varpi,\varrho)-\wp(\varpi',\varrho)\vert \vert \varphi_i(\varrho,\hbar(\varrho))\vert d\varrho\\
& \leq \sup_{i \geq 1}\int_0^\tau \vert \wp(\varpi,\varrho)-\wp(\varpi',\varrho)\vert  [ p_i(\varrho)+q_i(\varrho)\sup \{\vert \hbar_j(\varrho)\vert:j \geq i\} ]d\varrho\\
& \leq \sup_{i \geq 1}\int_0^\tau \vert \wp(\varpi,\varrho)-\wp(\varpi',\varrho)\vert  q_i(\varrho)\sup \{\vert \hbar_j(\varrho)\vert:j \geq i\} d\varrho\\
& \leq Q \int_0^\tau \vert \wp(\varpi,\varrho)-\wp(\varpi',\varrho)\vert \Vert \hbar(\varrho) \Vert_{c_0}ds.
\end{align*}
Since $\wp(\varpi,\varrho)$ is continuous, for given $\epsilon>0$ there exists $\delta(\varpi',\epsilon)>0$ such that $\vert \wp(\varpi,\varrho)-\wp(\varpi',\varrho)\vert<\frac{\epsilon}{Q\tau \Vert \hbar(s) \Vert_{c_0}}$ whenever $\vert \varpi-\varpi'\vert <\delta$. This implies that $\Vert (\aleph \hbar)(\varpi)-(\aleph \hbar)(\varpi')\Vert_{c_0}<\epsilon$ whenever $\vert \varpi-\varpi'\vert <\delta$. Similarly, we show that $\aleph(\mathfrak{M_2})$ is an equicontinuous subset of $\mathfrak{M_1}$. We further show that $\aleph$ is a condensing mapping. Let $(B_1,B_2) \subseteq (\mathfrak{M_1},\mathfrak{M_2})$ be a proximinal, nonempty, $\aleph$-invariant, convex, closed and bounded pair such that $dist(B_1,B_2)=dist(\mathfrak{M_1},\mathfrak{M_2})$. Consider
$$\Upsilon_{\mathcal{C}(I,c_0)}(\aleph(B_1)\cup \aleph(B_2))=\max\{\Upsilon_{\mathcal{C}(I,c_0)}(\aleph(B_1)),\Upsilon_{\mathcal{C}(I,c_0)}(\aleph(B_2))\}.$$
In view of Theorem \ref{thm1}, we have 
\begin{align*}
\Upsilon_{\mathcal{C}(I,c_0)}(\aleph(B_1))=\sup_{\varpi \in I} \Upsilon_{c_0}((\aleph B_1)(\varpi)).
\end{align*}
Using the relation established in (\ref{usec0}), we obtain 
\begin{align*}
\Upsilon_{c_0}((\aleph B_1)(\varpi))&=\lim_{n \rightarrow \infty} \big(\sup_{\hbar \in B_1}\big\{\max_{i \geq n}\vert (\aleph \hbar)_i(\varpi)\vert \big\} \big)\\
&\leq \lim_{n \rightarrow \infty} \Big(\sup_{\hbar \in B_1}\Big\{\max_{i \geq n}\Big\{\vert \lambda_i\vert +\int_0^\tau \vert \wp(\varpi,\varrho) \vert \vert \varphi_i(\varrho,\hbar(\varrho))\vert d\varrho \Big\}\Big\}\Big)\\
&\leq \lim_{n \rightarrow \infty} \Big(\sup_{\hbar \in B_1}\Big\{\max_{i \geq n}\Big\{\vert \lambda_i\vert +\int_0^\tau \vert \wp(\varpi,\varrho) \vert[ p_i(\varrho)+q_i(\varrho)\sup \{\vert \hbar_j(\varrho) \vert:j \geq i\} ] d\varrho \Big\}\Big\}\Big)\\
&\leq Q \lim_{n \rightarrow \infty} \Big(\sup_{\hbar \in B_1}\Big\{ \max_{i \geq n}\Big\{ \int_0^\tau \vert \wp(\varpi,\varrho) \vert \sup \{\vert \hbar_j(\varrho) \vert:j \geq i\} d\varrho \Big\}\Big\}\Big)
\end{align*}
which gives $$\sup_{\varpi \in I} \Upsilon_{c_0}((\aleph B_1)(\varpi)) \leq \frac{Q\tau^2}{8} \sup_{\varpi \in I}\lim_{n \rightarrow \infty} \Big( \sup_{\hbar \in B_1}\Big\{\max_{i \geq n} \vert \hbar_i(\varpi)\vert \Big\}\Big).$$
Therefore, $\Upsilon_{\mathcal{C}(I,c_0)}(\aleph(B_1)) \leq \frac{Q\tau^2}{8}\Upsilon_{\mathcal{C}(I,c_0)}(B_1)$. Similarly, $\Upsilon_{\mathcal{C}(I,c_0)}(\aleph(B_2)) \leq \frac{L\tau^2}{8}\Upsilon_{\mathcal{C}(I,c_0)}(B_2)$. This gives 
\begin{align*}
\Upsilon_{\mathcal{C}(I,c_0)}(\aleph(B_1)\cup \aleph(B_2))\leq \max\Big\{\frac{Q\tau^2}{8}\Upsilon_{\mathcal{C}(I,c_0)}(B_1),\frac{L\tau^2}{8}\Upsilon_{\mathcal{C}(I,c_0)}(B_2)\Big\}.
\end{align*}
Let $M=\max\{Q,L\}$. This implies that
\begin{align*}
\Upsilon_{\mathcal{C}(I,c_0)}(\aleph(B_1)\cup \aleph(B_2))&\leq \frac{M\tau^2}{8}\max\{\Upsilon_{\mathcal{C}(I,c_0)}(B_1),\Upsilon_{\mathcal{C}(I,c_0)}(B_2)\}\\
&=\frac{M\tau^2}{8}\Upsilon_{\mathcal{C}(I,c_0)}(B_1\cup B_2).
\end{align*}
Since $\frac{M\tau^2}{8}<1$, by applying Corollary \ref{corollary1}, we deduce  that $\aleph$ has a best proximity point, which, in turn, guarantees the existence of an optimal solution for system (\ref{equation3.1}).
\end{proof}
\subsection{Solution in $\ell_1$ space}
Let $\rho=(\rho_1,\rho_2,\rho_3,\ldots), \lambda=(\lambda_1,\lambda_2,\lambda_3,\ldots) \in \ell_1$. Let $\mathfrak{V_1}=\overline{\mathcal{B}}(\rho,\tilde{r}) \subseteq \ell_1$ and $\mathfrak{V_2}=\overline{\mathcal{B}}(\lambda,\tilde{r}) \subseteq \ell_1$. Let
$$\mathfrak{N_1}=\mathcal{C}(I,\mathfrak{V_1})=\{\hbar \in \mathcal{C}(I,\ell_1) \mbox{ such that }\hbar:I \rightarrow \mathfrak{V_1}, \hbar_i(0)=\hbar_i(\tau)=\rho_i\},$$
$$\mathfrak{N_2}=\mathcal{C}(I,\mathfrak{V_2})=\{\hbar \in \mathcal{C}(I,\ell_1) \mbox{ such that }\hbar:I \rightarrow \mathfrak{V_2}, \hbar_i(0)=\hbar_i(\tau)=\lambda_i\}.$$ Then $(\mathfrak{N_1},\mathfrak{N_2})$ is a nonempty, bounded, convex and closed pair in $\mathcal{C}(I,\ell_1)$ with $dist(\mathfrak{N_1},\mathfrak{N_2})=\Vert \rho-\lambda \Vert_{\ell_1}$.
Define $\aleph:\mathfrak{N_1} \cup \mathfrak{N_2} \rightarrow \mathfrak{N_1} \cup \mathfrak{N_2}$ by 
\begin{equation}\label{equation3.4}
\begin{aligned}
(\aleph \hbar)(\varpi)=\{(\aleph \hbar)_i(\varpi)\}=\left\{\left\{\begin{array}{ll} \lambda_i+\int_0^\tau \wp(\varpi,\varrho)\varphi_i(\varrho,\hbar(\varrho))d\varrho, & \hbar \in \mathfrak{N_1},\\
\rho_i+\int_0^\tau \wp(\varpi,\varrho)\chi_i(\varrho,\hbar(\varrho))d\varrho, & \hbar \in \mathfrak{N_2}. \end{array} \right\}\right., \quad \varpi \in I
\end{aligned}
\end{equation}
We say that $\upsilon \in \mathfrak{N_1} \cup \mathfrak{N_2}$ is an optimum solution of (\ref{equation3.1}) if the mapping given by (\ref{equation3.4}) possess a best proximity point $\upsilon$, that is, $\Vert \aleph\upsilon-\upsilon \Vert=dist(\mathfrak{N_1},\mathfrak{N_2})$. 

We establish our results under the following assumptions: 

($\mathfrak{A}_4$) The functions $\varphi(\varpi,\hbar(\varpi))=(\varphi_1(\varpi,\hbar(\varpi)),\varphi_2(\varpi,\hbar(\varpi)),\varphi_3(\varpi,\hbar(\varpi)),\ldots)$ and $\chi(\varpi,\hbar(\varpi))=(\chi_1(\varpi,\hbar(\varpi)),\chi_2(\varpi,\hbar(\varpi)),\chi_3(\varpi,\hbar(\varpi)),\ldots)$ are defined on $I \times \ell_1$ to $\ell_1$ and each $\varphi_i$ and $\chi_i$ are real valued functions such that 
$$\Vert \varphi(\varpi,\hbar(\varpi))-\chi(\varpi,\ell(\varpi)) \Vert_{c_0} \leq \frac{8}{\tau^2} [\Vert \hbar(\varpi)-\ell(\varpi) \Vert_{\ell_1}-\Vert \lambda-\rho \Vert_{\ell_1}].$$

($\mathfrak{A}_5$) $\varphi(.,\hbar)$ and $\chi(.,\hbar)$ are measurable for each fixed $\hbar$.

($\mathfrak{A}_6$)  For each $\varpi \in I$, $\hbar(\varpi)=(\hbar_i(\varpi)) \in \ell_1$ and for each $i \in \mathbb{N}$, we have
\begin{align*}
\vert \varphi_i(\varpi,\hbar(\varpi))\vert &\leq r_i(\varpi)+d_i(\varpi)\{\vert \hbar_i (\varpi)\vert\},\\
\vert \chi_i(\varpi,\hbar(\varpi))\vert &\leq y_i(\varpi)+z_i(\varpi) \{\vert \hbar_i(\varpi)\vert\},
\end{align*}
where $r_i(\varpi)$, $d_i(\varpi)$, $y_i(\varpi)$ and $z_i(\varpi)$ are continuous functions on $I$ such that the series $\sum\limits_{i\geq1} r_i(\varpi)$ and $\sum\limits_{i\geq1}y_i(\varpi)$ converges uniformly on $I$ and the sequences $(d_i(\varpi))$ and $(z_i(\varpi))$ are equibounded on $I$.

We adopt the following notation:
\begin{align*}
r(\varpi)&=\sum_{i\geq1}r_i(\varpi)  \text{ for each }  \varpi\in I,\\
R&=\sup \{r(\varpi):\varpi \in I\},\\
D&=\sup \{d_i(\varpi):\varpi \in I, i \in \mathbb{N}\},\\
y(\varpi)&=\sum_{i\geq1}y_i(\varpi)  \text{ for each }  \varpi \in I,\\
Y&=\sup \{y(\varpi):\varpi \in I\},\\
Z&=\sup \{z_i(\varpi):\varpi \in I, i \in \mathbb{N}\}.
\end{align*}
\begin{theorem}
Along with the assumptions $(\mathfrak{A}_4)$-$(\mathfrak{A}_6)$ and $\frac{D\tau^2}{8}<1$ and $\frac{Z\tau^2}{8}<1$, the infinite system of second order differential equations (\ref{equation3.1}) has an optimum solution.
\end{theorem}
\begin{proof}
Consider
\begin{align*}
\Vert (\aleph \hbar)(\varpi)-\lambda \Vert _{\ell_1}&=\sum_{i=1}^{\infty}\vert (\aleph \hbar)_i(\varpi)-\lambda_i \vert \\
&=\sum_{i=1}^{\infty} \Big\vert \lambda_i+\int_0^\tau \wp(\varpi,\varrho)\varphi_i(\varrho,\hbar(\varrho))d\varrho-\lambda_i \Big\vert\\
& \leq \sum_{i=1}^{\infty}\int_0^\tau \vert \wp(\varpi,\varrho) \vert \vert \varphi_i(\varrho,\hbar(\varrho))\vert d\varrho \\
& \leq \sum_{i=1}^{\infty}\int_0^\tau \vert \wp(\varpi,\varrho) \vert [ r_i(\varrho)+d_i(\varrho)\vert \hbar_i(\varrho)\vert ]d\varrho\\
& \leq \int_0^\tau \sum_{i=1}^{\infty}   \vert \wp(\varpi,\varrho)\vert   r_i(\varrho)ds+\int_0^\tau\sum_{i=1}^{\infty}   \vert \wp(\varpi,\varrho) \vert d_i(\varrho)\vert \hbar_i(\varrho)\vert d\varrho\\
& \leq \frac{\tau^2}{8}[R+D\Vert \hbar(\varpi) \Vert _{\ell_1}]. 
\end{align*}
As $\Vert \hbar(\varpi) \Vert _{\ell_1}\leq  \Vert \hbar(\varpi)-\rho \Vert_{\ell_1}+\Vert \rho \Vert _{\ell_1}$, this gives 
$\Vert (\aleph \hbar)(\varpi)-\lambda \Vert _{\ell_1} < \frac{\tau^2}{8}[R_1+D\Vert \hbar(\varpi) -\rho \Vert_{\ell_1}]$, where $R_1=R+D\Vert \rho \Vert_{\ell_1}$. Then $r_3=\frac{\tau^2R_1}{8-\tau^2D}$ is the optimal solution of this inequality
$$\frac{\tau^2}{8}[R_1+Dr] \leq r.$$
Consider
\begin{align*}
\Vert (\aleph \ell)(\varpi)-\rho \Vert _{\ell_1}&=\sum_{i=1}^{\infty}\vert (\aleph \ell)_i(\varpi)-\rho_i \vert \\
&=\sum_{i=1}^{\infty} \Big\vert \rho_i+\int_0^\tau \wp(\varpi,\varrho)\chi_i(\varrho,\ell(\varrho))d\varrho-\rho_i \Big\vert\\
& \leq \sum_{i=1}^{\infty}\int_0^\tau \vert \wp(\varpi,\varrho) \vert \vert \chi_i(\varrho,\ell(\varrho))\vert d\varrho \\
& \leq \sum_{i=1}^{\infty}\int_0^\tau \vert \wp(\varpi,\varrho) \vert [ y_i(\varrho)+z_i(\varrho)\{\vert \ell_i(\varrho)\vert\} ]d\varrho\\
& \leq \int_0^\tau  \sum_{i=1}^{\infty}  \vert \wp(\varpi,\varrho)\vert   y_i(\varrho)d\varrho+\int_0^\tau \sum_{i=1}^{\infty}  \vert \wp(\varpi,\varrho) \vert z_i(\varrho)\vert \ell_i(\varrho)\vert d\varrho\\
& \leq \frac{\tau^2}{8}[Y+Z\Vert \ell(\varpi) \Vert_{\ell_1}]. 
\end{align*}
As $\Vert \ell(\varpi) \Vert _{\ell_1}\leq  \Vert \ell(\varpi)-\lambda \Vert_{\ell_1}+\Vert \lambda \Vert _{\ell_1}$, this gives 
$\Vert (\aleph \ell)(\varpi)-\rho \Vert_{\ell_1} < \frac{\tau^2}{8}[Y_1+Z\Vert \ell -\lambda\Vert_{\ell_1}]$, where $Y_1=Y+Z\Vert \lambda \Vert_{\ell_1}$. Then $r_4=\frac{\tau^2Y_1}{8-\tau^2Z}$ is the optimal solution of this inequality
$$\frac{\tau^2}{8}[Y_1+Zr] \leq r.$$
Let $\tilde{r}=\max\{r_3,r_4\}$. Now, we show that $\aleph$ is a cyclic mapping on $\mathfrak{N_1} \cup \mathfrak{N_2}$.  First, we prove that $\aleph(\mathfrak{N_1}) \subseteq \mathfrak{N_2}$. Let $\hbar \in \mathfrak{N_1}$. Consider 
\begin{align*}
(\aleph \hbar)_i(0)&=\lambda_i+\int_0^\tau \wp(0,\varrho)\varphi_i(\varrho,\hbar(\varrho))d\varrho\\
&=\lambda_i+\int_0^\tau 0.\varphi_i(\varrho,\hbar(\varrho))d\varrho\\
&=\lambda_i
\end{align*} 
and 
\begin{align*}
(\aleph \hbar)_i(\tau)&=\lambda_i+\int_0^\tau \wp(\tau,\varrho)\varphi_i(\varrho,\hbar(\varrho))d\varrho\\
&=\lambda_i+\int_0^\tau 0.\varphi_i(\varrho,\hbar(\varrho))d\varrho\\
&=\lambda_i.
\end{align*}
Also, $\Vert (\aleph \hbar)(\varpi)-\lambda \Vert _{\ell_1}\leq r$. This implies that $\aleph \hbar \in \mathfrak{N_2}$. Now, we show that $\aleph(\mathfrak{N_2}) \subseteq \mathfrak{N_1}$. Let $\ell \in \mathfrak{N_2}$. Consider 
\begin{align*}
(\aleph \ell)_i(0)&=\rho_i+\int_0^\tau \wp(0,\varrho)\chi_i(\varrho,\ell(\varrho))d\varrho\\
&=\rho_i+\int_0^\tau 0.\chi_i(\varrho,\ell(\varrho))d\varrho\\
&=\rho_i
\end{align*} 
and 
\begin{align*}
(\aleph \ell)_i(\tau)&=\rho_i+\int_0^\tau \wp(\tau,\varrho)\chi_i(\varrho,\ell(\varrho))d\varrho\\
&=\rho_i+\int_0^\tau 0.\chi_i(\varrho,\ell(\varrho))d\varrho\\
&=\rho_i.
\end{align*}
Also, $\Vert (\aleph \ell)(\varpi)-\rho \Vert _{\ell_1}\leq r$. This implies that $\aleph \ell \in \mathfrak{N_1}$. Therefore, $\aleph$ is a cyclic mapping. Clearly, $\aleph(\mathfrak{N_1})$ and $\aleph(\mathfrak{N_2})$ are  bounded subsets of $\mathcal{C}(I,\ell_1)$. We proceed to establish that $\aleph$ is a relatively nonexpansive mapping. Let $\hbar \in \mathfrak{N_1}$ and $\ell \in \mathfrak{N_2}$. Consider 
\begin{align*}
\Vert (\aleph \hbar)(\varpi)-(\aleph \ell)(\varpi) \Vert_{\ell_1}&=\sum_{i=1}^{\infty} \Big\vert \lambda_i+\int_0^\tau \wp(\varpi,\varrho)\varphi_i(\varrho,\hbar(\varrho))d\varrho-\rho_i-\int_0^\tau \wp(\varpi,\varrho)\chi_i(\varrho,\ell(\varrho))d\varrho \Big\vert \\
& \leq \sum_{i=1}^{\infty} \Big[\vert \lambda_i-\rho_i \vert +\int_0^\tau \vert \wp(\varpi,\varrho)\vert \vert \varphi_i(\varrho,\hbar(\varrho))-\chi_i(\varrho,\ell(\varrho))\vert d\varrho \Big]\\
& \leq \sum_{i=1}^{\infty}\vert \lambda_i-\rho_i \vert + \sum_{i=1}^{\infty}\int_0^\tau \vert \wp(\varpi,\varrho)\vert \vert \varphi_i(\varrho,\hbar(\varrho))-\chi_i(\varrho,\ell(\varrho))\vert d\varrho \\
& \leq \sum_{i=1}^{\infty}\vert \lambda_i-\rho_i \vert +  \int_0^\tau \vert \wp(\varpi,\varrho)\vert \sum_{i=1}^{\infty}\vert \varphi_i(\varrho,\hbar(\varrho))-\chi_i(\varrho,\ell(\varrho))\vert d\varrho \\
& = \sum_{i=1}^{\infty}\vert \lambda_i-\rho_i \vert +  \int_0^\tau \vert \wp(\varpi,\varrho)\vert \Vert \varphi(\varrho,\hbar(\varrho))-\chi(\varrho,\ell(\varrho))\Vert_{\ell_1} d\varrho \\
&\leq \Vert \lambda-\rho\Vert_{\ell_1}+\frac{8}{\tau^2}\int_0^\tau \vert \wp(\varpi,\varrho)\vert [\Vert \hbar-\ell \Vert_{\mathcal{C}(I,\ell_1)}-\Vert \lambda-\rho\Vert_{\ell_1}]ds\\
&\leq \Vert \hbar-\ell \Vert_{\mathcal{C}(I,\ell_1)}
\end{align*}
which implies that $\Vert \aleph \hbar-\aleph \ell \Vert_{\mathcal{C}(I,\ell_1)} \leq \Vert \hbar-\ell \Vert_{\mathcal{C}(I,\ell_1)}$. 
Now, we show that $\aleph(\mathfrak{N_1})$ is an  equicontinuous subset of $\mathfrak{N_2}$. Consider 
\begin{align*}
\Vert (\aleph \hbar)(\varpi)-(\aleph \hbar)(\varpi')\Vert_{\ell_1}&=\sum_{i=1}^{\infty} \Big\vert \lambda_i+\int_0^\tau \wp(\varpi,\varrho)\varphi_i(\varrho,\hbar(\varrho))d\varrho-\lambda_i-\int_0^\tau \wp(\varpi',\varrho)\varphi_i(\varrho,\hbar(\varrho))d\varrho \Big\vert\\
& \leq \sum_{i=1}^{\infty}\int_0^\tau \vert \wp(\varpi,\varrho)-\wp(\varpi',\varrho)\vert \vert \varphi_i(\varrho,\hbar(\varrho))\vert d\varrho\\
& \leq \sum_{i=1}^{\infty}\int_0^\tau \vert \wp(\varpi,\varrho)-\wp(\varpi',\varrho)\vert [[r_i(\varrho)+d_i(\varrho)\vert \hbar_i(\varrho)\vert] d\varrho \\
& \leq \sum_{i=1}^{\infty}\int_0^\tau \vert \wp(\varpi,\varrho)-\wp(\varpi',\varrho)\vert   r_i(\varrho)d\varrho+\sum_{i=1}^{\infty}\int_0^\tau \vert \wp(\varpi,\varrho)-\wp(\varpi',\varrho)\vert d_i(\varrho)\vert \hbar_i(\varrho)\vert d\varrho
\end{align*}
\begin{align*}
& \hspace{3cm} \leq \sum_{i=1}^{\infty}\int_0^\tau \vert \wp(\varpi,\varrho)-\wp(\varpi',\varrho)\vert   r_i(\varrho)d\varrho+\sum_{i=1}^{\infty}D \int_0^\tau \vert \wp(\varpi,\varrho)-\wp(\varpi',\varrho)\vert \vert \hbar_i(\varrho)\vert d\varrho\\
&\hspace{3cm}\leq \int_0^\tau \sum_{i=1}^{\infty}\vert \wp(\varpi,\varrho)-\wp(\varpi',\varrho)\vert   r_i(\varrho)d\varrho+D\int_0^\tau \sum_{i=1}^{\infty}\vert \wp(\varpi,\varrho)-\wp(\varpi',\varrho)\vert\vert \hbar_i(\varrho)\vert d\varrho\\
&\hspace{3cm} \leq  R \int_0^\tau \vert \wp(\varpi,\varrho)-\wp(\varpi',\varrho)\vert d\varrho+  D \int_0^\tau \vert \wp(\varpi,\varrho)-\wp(\varpi',\varrho)\vert \Vert \hbar(\varrho) \Vert_{\ell_1} d\varrho.
\end{align*}
Since $\wp(\varpi,\varrho)$ is continuous, for given $\epsilon>0$ there exists $\delta(\varpi',\epsilon)>0$ such that $\vert \wp(\varpi,\varrho)-\wp(\varpi',\varrho)\vert<\frac{\epsilon}{2D\tau \Vert \hbar(s) \Vert_{\ell_1}}$ and  $\vert \wp(\varpi,\varrho)-\wp(\varpi',\varrho)\vert<\frac{\epsilon}{2R}$  whenever $\vert \varpi-\varpi'\vert <\delta$. This implies that $\Vert (\aleph \hbar)(\varpi)-(\aleph \hbar)(\varpi')\Vert_{\ell_1}<\epsilon$ whenever $\vert \varpi-\varpi'\vert <\delta$. Similarly, we show that $\aleph(\mathfrak{N_2})$ is an equicontinuous subset of $\mathfrak{N_1}$. We further show that $F$ is a condensing mapping. Let $(B_1,B_2) \subseteq (\mathfrak{N_1},\mathfrak{N_2})$ be a proximinal, nonempty, $\aleph$-invariant, convex, closed and bounded pair such that $dist(B_1,B_2)=dist(\mathfrak{N_1},\mathfrak{N_2})$. Consider
$$\Upsilon_{\mathcal{C}(I,\ell_1)}(\aleph(B_1)\cup \aleph(B_2))=\max\{\Upsilon_{\mathcal{C}(I,\ell_1)}(\aleph(B_1)),\Upsilon_{\mathcal{C}(I,\ell_1)}(\aleph(B_2))\}.$$
In view of Theorem \ref{thm1}, we have 
\begin{align*}
\Upsilon_{\mathcal{C}(I,\ell_1)}(\aleph(B_1))=\sup_{\varpi \in I} \Upsilon_{\ell_1}((\aleph B_1)(\varpi))
\end{align*}
Using the relation established in (\ref{usel1}), we obtain  
\begin{align*}
\Upsilon_{\ell_1}((\aleph B_1)(\varpi))&=\lim_{n \rightarrow \infty} \big(\sup_{\hbar \in B_1}\big\{\sum_{i \geq n}\vert (\aleph \hbar)_i(\varpi)\vert \big\} \big)\\
&\leq \lim_{n \rightarrow \infty} \Big(\sup_{\hbar \in B_1}\Big\{\sum_{i \geq n}\Big(\vert \lambda_i\vert +\int_0^\tau \vert \wp(\varpi,\varrho) \vert \vert \varphi_i(\varrho,\hbar(\varrho))\vert d\varrho\Big)\Big\}\Big)\\
&\leq \lim_{n \rightarrow \infty} \Big(\sup_{\hbar \in B_1}\Big\{\sum_{i \geq n}\Big( \vert \lambda_i\vert +\int_0^\tau \vert \wp(\varpi,\varrho) \vert[ r_i(\varrho)+d_i(\varrho)\vert \hbar_i(\varrho) \vert ] d\varrho\Big)\Big\}\Big)\\
&= \lim_{n \rightarrow \infty} \Big(\sup_{\hbar \in B_1}\Big\{\sum_{i \geq n}\Big( \vert \lambda_i\vert +\int_0^\tau \vert \wp(\varpi,\varrho)\vert  r_i(\varrho)d\varrho+\int_0^\tau \vert \wp(t,\varrho)\vert d_i(\varrho)\vert \hbar_i(\varrho) \vert d\varrho \Big) \Big\}\Big)\\
&\leq \lim_{n \rightarrow \infty} \Big(\sup_{\hbar \in B_1}\Big\{\sum_{i \geq n} \int_0^\tau \vert \wp(t,\varrho)\vert r_i(\varrho)d\varrho+\sum_{i \geq n}\int_0^\tau  \vert \wp(\varpi,\varrho)\vert d_i(\varrho)\vert \hbar_i(\varrho) \vert d\varrho\Big\}\Big)\\
&\leq \lim_{n \rightarrow \infty} \Big(\sup_{\hbar \in B_1}\Big\{\int_0^\tau \vert \wp(\varpi,\varrho) \vert \sum_{i \geq n}  r_i(\varrho)d\varrho+D \int_0^\tau  \vert \wp(\varpi,\varrho)\vert \sum_{i \geq n} \vert \hbar_i(\varrho) \vert d\varrho \Big\}\Big)\\
&\leq D \lim_{n \rightarrow \infty}\Big(\sup_{\hbar \in B_1}\Big\{ \int_0^\tau \vert \wp(\varpi,\varrho) \vert\sum_{i \geq n}\vert \hbar_i(\varrho) \vert d\varrho \Big\}\Big)
\end{align*}
which gives $$\sup_{\varpi \in I} \Upsilon_{\ell_1}((\aleph B_1)(\varpi)) \leq \frac{D\tau^2}{8} \sup_{\varpi \in I} \lim_{n \rightarrow \infty}\Big(\sup_{\hbar \in B_1}\Big\{ \sum_{i \geq n}\vert \hbar_i(\varpi) \vert \Big\} \Big).$$
Therefore, $\Upsilon_{\mathcal{C}(I,\ell_1)}(\aleph(B_1)) \leq \frac{D\tau^2}{8}\Upsilon_{\mathcal{C}(I,\ell_1)}(B_1)$. Similarly, $\Upsilon_{\mathcal{C}(I,\ell_1)}(\aleph(B_2)) \leq \frac{Z\tau^2}{8}\Upsilon_{\mathcal{C}(I,\ell_1)}(B_2)$. This gives 
\begin{align*}
\Upsilon_{\mathcal{C}(I,\ell_1)}(\aleph(B_1)\cup \aleph(B_2))\leq \max\Big\{\frac{D\tau^2}{8}\Upsilon_{\mathcal{C}(I,\ell_1)}(B_1),\frac{Z\tau^2}{8}\Upsilon_{\mathcal{C}(I,\ell_1)}(B_2)\Big\}.
\end{align*}
Let $\tilde{M}=\max\{D,Z\}$. This implies that
\begin{align*}
\Upsilon_{\mathcal{C}(I,\ell_1)}(\aleph(B_1)\cup \aleph(B_2))&\leq \frac{\tilde{M}\tau^2}{8}\max\{\Upsilon_{\mathcal{C}(I,\ell_1)}(B_1),\Upsilon_{\mathcal{C}(I,\ell_1)}(B_2)\}\\
&=\frac{\tilde{M}\tau^2}{8}\Upsilon_{\mathcal{C}(I,\ell_1)}(B_1\cup B_2).
\end{align*}
Since $\frac{\tilde{M}\tau^2}{8}<1$, by applying Corollary \ref{corollary1}, we deduce  that $\aleph$ has a best proximity point, which, in turn, guarantees the existence of an optimal solution for system (\ref{equation3.1}).
\end{proof}
\section*{Acknowledgements}
The first author is grateful to the Government of India for providing financial support in the form of Institutional Fellowship.   

\end{document}